\documentclass[11pt]{amsart}
\usepackage{amsfonts,amssymb,verbatim,amsmath,amsthm,latexsym,textcomp,amscd,hyperref}
\usepackage{latexsym,amssymb,epsfig,verbatim}
\usepackage{amsmath,amsthm,amssymb,latexsym,graphics,textcomp}
\usepackage{blindtext}
\usepackage[left=3cm,right=3cm,top=2.86cm,bottom=2.86cm]{geometry}
\usepackage{graphicx}
\usepackage{color}
\usepackage{url}
\usepackage{pgfplots}
\usepackage{tikz}
\usepackage{tikz-cd}
\usepackage[utf8]{inputenc}
\usepackage{mathscinet}

\begin{document}
	\newtheorem{theorem}{Theorem}[section]
	\newtheorem{prop}[theorem]{Proposition}
	\newtheorem{lemma}[theorem]{Lemma}
	\newtheorem{cor}[theorem]{Corollary}
	\newtheorem{qn}[theorem]{Question}
	\theoremstyle{remark}
	\theoremstyle{definition}
	
	\newtheorem{prob}[theorem]{Problem}
	\newtheorem{defn}[theorem]{Definition}
	\newtheorem{notation}[theorem]{Notation}
	\newtheorem{fact}[theorem]{Fact}
	\newtheorem{conj}[theorem]{Conjecture}
	\newtheorem{claim}[theorem]{Claim}
	\newtheorem{example}[theorem]{Example}
	\newtheorem{rem}[theorem]{Remark}
	\newtheorem{assumption}[theorem]{Assumption}
	\newtheorem{scholium}[theorem]{Scholium}
	
	\newtheorem{conv}[theorem]{Notation and Convention}
	
	\newcommand{\HAT}{\widehat}
	\newcommand{\map}{\rightarrow}
	\newcommand{\C}{\mathcal C}
	\newcommand\AAA{{\mathcal A}}
	\def\AA{\mathcal A}
	
	\def\L{{\mathcal L}}
	\def\al{\alpha}
	\def\A{{\mathcal A}}

	\newcommand\GB{{\mathbb G}}
	\newcommand\BB{{\mathcal B}}
	\newcommand\DD{{\mathcal D}}
	\newcommand\EE{{\mathcal E}}
	\newcommand\FF{{\mathcal F}}
	\newcommand\GG{{\mathcal G}}
	\newcommand\HH{{\mathbb H}}
	\newcommand\II{{\mathcal I}}
	\newcommand\JJ{{\mathcal J}}
	\newcommand\KK{{\mathcal K}}
	\newcommand\LL{{\mathcal L}}
	\newcommand\MM{{\mathcal M}}
	\newcommand\NN{{\mathbb N}}
	\newcommand\OO{{\mathcal O}}
	\newcommand\PP{{\mathcal P}}
	\newcommand\QQ{{\mathbb Q}}
	\newcommand\RR{{\mathbb R}}
	\newcommand\SSS{{\mathcal S}}
	\newcommand\TT{{\mathcal T}}
	\newcommand\UU{{\mathcal U}}
	\newcommand\VV{{\mathcal V}}
	\newcommand\WW{{\mathcal W}}
	\newcommand\XX{{\mathcal X}}
	\newcommand\YY{{\mathcal Y}}
	\newcommand\YB{{\mathbb Y}}
	\newcommand\ZZ{{\mathcal Z}}
	\newcommand\ZI{{\mathbb Z}}
	\newcommand\hhat{\widehat}
	\newcommand\flaring{{Corollary \ref{cor:super-weak flaring} }}
	\newcommand\pb{\bar{p}_B}
	\newcommand\pp{\bar{p}_{B_1}}
	\newcommand{\eg}{\overline{EG}}
	\newcommand{\eh}{\overline{EH}}
	\def\Ga{\Gamma}
	\def\Z{\mathbb Z}
	
	\def\diam{\operatorname{diam}}
	\def\dist{\operatorname{dist}}
	\def\hull{\operatorname{Hull}}
	\def\id{\operatorname{id}}
	\def\Im{\operatorname{Im}}
	
	\def\barycenter{\operatorname{center}}

	\def\length{\operatorname{length}}
	\newcommand\RED{\textcolor{red}}
	\newcommand\BLUE{\textcolor{blue}}
	\newcommand\GREEN{\textcolor{green}}
	\def\mini{\scriptsize}
	
	\def\acts{\curvearrowright}
	\def\embed{\hookrightarrow}
	
	\def\ga{\gamma}
	\newcommand\la{\lambda}
	\newcommand\eps{\epsilon}
	\def\geo{\partial_{\infty}}
	\def\bhb{\bigskip\hrule\bigskip}

	\title[Homeomorphism types of Floyd boundaries]{Homeomorphism types of Floyd boundaries of infinite-ended groups}

	\author{Subhajit Chakraborty and Ravi Tomar}
	\email{ravitomar547@gmail.com}
	\address{Department of Mathematical Sciences,
		Indian Institute of Science Education and Research Bhopal,
		Bhopal Bypass Road, Bhauri 462066
		Madhya Pradesh, India}
	
	\email{scac0612@gmail.com}


	
	\date{\today}
	
	\begin{abstract} 
	Suppose $G$ is a finitely generated infinite group and $\mathcal G$ is a graph of groups decomposition of $G$ such that the edge groups are finite. This paper establishes that the topology of the Floyd boundary of $G$ is uniquely determined by the topology of the Floyd boundary of each vertex group of $\mathcal G$.
	\end{abstract}
	\maketitle
	\section{Introduction}
One of the first compactification of groups was introduced by W. Floyd in \cite{floyd}. It is now known as Floyd boundary. In general, the Floyd boundary depends on a scaling function known as Floyd function.  This compactification exhibits a strong connection to the theory of hyperbolic and relatively hyperbolic groups. Given a hyperbolic group $G$, there exists a Floyd function $f$ such that Floyd boundary of $G$ with respect to $f$ is homeomorphic to Gromov boundary of $G$ (see \cite{gromov-hypgps}). Similarly, when $G$ is a finitely generated group hyperbolic relative to a collection of subgroups, Gerasimov in \cite{gerasimov-floyd-map} shows that there exists a Floyd function $f$ and a continuous $G$-equivariant map from Floyd boundary of $G$ with respect to $f$ to Bowditch boundary of $G$. Floyd boundary also has connection to convergence actions (see \cite{karlsson}).

In \cite{martin-swiat}, Martin--\`{S}wi\polhk{a}tkowski proved that the topology of Gromov boundary of a free product of hyperbolic groups is uniquely determined by the topology of the Gromov bounary of each free factor. In this paper, we drop the hypothesis that each free factor is hyperbolic and prove a similar result for Floyd boundary. Throughout the paper, all groups are assumed to be finitely generated. The following is the main result of this paper:

\vspace{.2cm}
{\bf Theorem 1.} (Theorem \ref{general case}) {\em Suppose $G$ and $H$ are two infinite groups. Suppose $\mathcal G$ and $\mathcal H$ are graphs of groups decomposition of $G$ and $H$, respectively such that the edge groups of $\mathcal G$ and $\mathcal H$ are finite. Then, we have the following:
	\begin{enumerate}
		\item If each vertex group of $\mathcal G$ is elementary and $G$ has infinitely many ends then the Floyd boundary of $G$ is homeomorphic to the Cantor set.
		\item Suppose at least one vertex group of either $\mathcal G$ or $\mathcal H$ is non-elementary. Let $h(\mathcal G)$ and $h(\mathcal H)$ denote the set of homeomorphism types of Floyd boundaries of non-elementary vertex groups of $\mathcal G$ and $\mathcal H$, respectively. If $h(\mathcal G)=h(\mathcal H)$ then the Floyd boundary of $G$ is homeomorphic to Floyd boundary of $H$.
\end{enumerate}}

A non-elementary group is one that is not virtually cyclic. One interesting consequence of this theorem is that, for any Floyd function, the Floyd boundary of a virtually free group is homeomorphic to the Cantor set (see Corollary \ref{virtually free cor}).  We need the following result to prove the previous theorem, which tells us the homeomorphism type of the Floyd boundary of a free product.

\vspace{.2cm}
{\bf Theorem 2.} (Theorem \ref{main theorem}) {\em For $n,m\geq 2$, let $G_1=A_1\ast\dots \ast A_n$ and $G_2=B_1\ast\dots \ast B_m$ be two free products where at least one free factor of either $G_1$ or $G_2$ is non-elementary. Let $h(G_1)$ and $h(G_2)$ denote the set of homeomorphism types of Floyd boundaries of non-elementary free factors of $G_1$ and $G_2$, respectively. If $h(G_1)=h(G_2)$ then the Floyd boundary of $G_1$ is homeomorphic to the Floyd boundary of $G_2$.}

\vspace{.2cm}
{\bf Remark.}
	In the above theorem, if each free factor of $G_1$ is elementary then either Floyd boundary of $G_1$ contains $2$ points or it is homeomorphic to the Cantor set (see Corollary \ref{homeo free product cyclic}).
	
	\vspace{.2cm}

Next, we look at the connected components of the Floyd boundary of an infinite-ended group and obtain a partial converse to Theorem 1.
 \vspace{.2cm}
 
 {\bf Theorem 3.} (Theorem \ref{partial converse}) {\em Suppose $G$ and $H$ are two finitely generated infinite-ended groups. Suppose $\mathcal G$ and $\mathcal H$ are graphs of groups decomposition of $G$ and $H$, respectively such that the edge groups of $\mathcal G$ and $\mathcal H$ are finite, and each non-elementary vertex group of $\mathcal G$ and $\mathcal H$ is $1$-ended. Also, assume that at least one vertex group of $\mathcal G$ and $\mathcal H$ is non-elementary. Then, we have the following:
 	
 	Let $h(\mathcal G)$ and $h(\mathcal H)$ denote the set of homeomorphism types of Floyd boundaries of non-elementary vertex groups of $\mathcal G$ and $\mathcal H$, respectively. If Floyd boundary of $G$ is homeomorphic to Floyd boundary of $H$ then $h(\mathcal G)=h(\mathcal H)$.} 
\begin{comment}
	
\vspace{.2cm}
With an additional assumption on the Floyd functions, we generalize the above theorem for graphs of groups over finite groups.

\vspace{.2cm}
{\bf Theorem 2.} {\em Suppose $G$ and $H$ are two infinite-ended groups and $f$ is a Floyd function. Let $\mathcal G$ and $\mathcal H$ be graphs of groups decomposition of $G$ and $H$ such that the following hold:
	\begin{enumerate}
		\item All the vertex groups of $\mathcal G$ and $\mathcal H$ are non-elementary (i.e. not virtually cyclic).
		\item All the edge groups of $\mathcal G$ and $\mathcal H$ are finite.
		\item There exists a constant $D>0$ such that $f(n)/f(2n)\leq D$.
	\end{enumerate}
Let $h(\mathcal G)$ and $h(\mathcal H)$ denote the set of homeomorphism types of Floyd boundaries of vertex groups of $\mathcal G$ and $\mathcal H$, respectively. If $h(\mathcal G)=h(\mathcal H)$ then $\partial_f(G)$ is homeomorphic to $\partial_f(H)$.}

\end{comment}
\vspace{.2cm}

{\bf Organization of the paper:} In Section \ref{2}, we recall some basic facts about Floyd boundary and exlain Cayley graphs of amalgamted free products and HNN extensions. Section \ref{3} contains a construction of Floyd boundaries of the amalgamated free products and HNN extensions over finite groups. In Section \ref{4}, we discuss the homeomorphism type of Floyd boundaries of free products. Then, in Section \ref{5}, we discuss homeomorphism types of Floyd boundaries of amalgamated free product and HNN extension. There, we deduce a number of interesting corollaries too. In Section \ref{6}, we give a proof of Theorem \ref{general case}. Finally, in Section \ref{7}, we finish the paper by discussing the connected components of Floyd boundary of a graph of groups with finite edge groups.
\section{Preliminaries}\label{2}
\subsection{Coarse geometric notion} Let $X$ be a metric space and let $x,y\in X$. A {\em geodesic} joining $x$ and $y$ is an isometric embedding $\alpha:[a,b]\subset\mathbb R\to X$ such that $\alpha(a)=x$ and $\alpha(b)=y$. If any two points of $X$ can be joined by a geodesic then $X$ is called a {\em geodesic metric space}. A geodesic joining $x$ and $y$ will be denoted by $[x,y]$. Throughout the paper, all graphs are assumed to be connected. For $z\in X$ and $A\subset X$, we shall denote by $d(z,A)$ the quantity $\inf\{d(z,a):a\in A\}$. 

Suppose $(X,d_X)$ and $(Y,d_Y)$ are two metric spaces. Given $L\geq 0$, an $L$-{\em bi-Lipschitz} map $\phi:X\to Y$ is one such that $\dfrac{1}{L}d_X(x,x')\leq d_Y(\phi(x),\phi(x'))\leq L d_X(x,x').$ The map $\phi$ is said to be bi-Lipschitz if it is $L$-bi-Lipschitz for some $L\geq 0$.
 \begin{comment}
	
Given $\lambda\geq 1,\epsilon\geq 0$, a map 
$\phi:X\rightarrow Y$ is said to be a \emph{$(k,\epsilon)$-quasiisometric embedding} if for all $x,x'\in X$ we have,
$$\dfrac{1}{k}d_X(x,x')-\epsilon\leq d_Y(\phi(x),\phi(x'))\leq k d_X(x,x')+\epsilon.$$
The map $\phi$ is said to be $(k,\epsilon)$-\emph{quasiisometry} if $\phi$ is a \emph{$(k,\epsilon)$-quasiisometric embedding} 
and moreover, $N_D(f(X))=Y$ for some $D\geq 0$. The map $\phi$ is said to be {\em quasisometry} if it is $(k,\epsilon)$-quasiisometry for some $k\geq 1,\epsilon\geq 0$.
\end{comment}

\subsection{Floyd completion of groups}
The notion of Floyd boundary of a group was introduced by W. Floyd in \cite{floyd}. Let $G$ be a group generated by a finite set, say, $S$. Let $\Gamma_{G,S}$ be the Cayley graph of $G$ with respect to $S$. By assigning each edge a unit length, we see that $\Gamma_{G,S}$ is naturally a complete geodesic metric space. We denote this metric by $d$. Let $f:\mathbb N\to \mathbb{R}_{> 0}$ be a function satisfying the following two conditions:

(1) $\Sigma_{n=1}^{\infty}f(n)<\infty$.

(2) There exists $\lambda>0$ such that $f(n)\geq f(n+1)\geq \lambda f(n)$ for all $n\in\mathbb N$.

We call $f$ a {\em Floyd function.} To simplify the construction of the Floyd boundary, for any Floyd function f, we deﬁne $f(0):=f(1).$ An example of a Floyd function is $f(n)={1}/{n^2}$. We define a new metric on $\Gamma_{G,S}$ by scaling the length of edges. Let $e_G$ denote the identity element of $G$. For an edge $e$ in $\Gamma_{G,S}$ connecting the vertices $g,h$, we define the new edge length $l_f(e)$ of $e$ to be $f(d(e_G,\{g,h\}))$. If $\alpha$ is a path in $\Gamma_{G,S}$, given by the consecutive edges $e_1,e_2,\dots ,e_n$, then the {\em Floyd length} of $\alpha$ is defined as $\Sigma_{i=1}^nl_f(e_i)$. For $v,w\in \Gamma_{G,S}$, define $$d_f(v,w)=\inf\{l_f(\alpha):\alpha \text{ is a path in $\Gamma_{G,S}$ joining $v$ and $w$}\}.$$ It is easy to see that $d_f$ is a metric and it is called the {\em Floyd metric} on $\Gamma_{G,S}$.

\begin{defn}
	[Floyd boundary] The Cauchy completion $\overline{\Gamma}_{G,S}$ of the metric space $(\Gamma_{G,S},d_f)$ is called the {\em Floyd completion.} The subspace $\overline{\Gamma}_{G,S}\setminus\{\Gamma_{G,S}\}$ is called {\em Floyd boundary} of $G$. It is denoted by $\partial_f(G)$.
\end{defn}
This definition does not depend on the choice of finite generating set $S$ of $G$ (see \cite[Lemma 2]{floyd}). Thus, we suppress the dependence on the generating set and denote the completion by $\overline{\Gamma}_G$. Note that (1) $\Gamma_{G}$ has finite $d_f$-diameter as $f$ is summable. (2) $\overline{\Gamma}_{G}$ is compact. Since the group $G$ acts by isometries on $\Gamma_{G}$, this action extends to an action of $G$ on $\overline{\Gamma}_{G}$ by homeomorphism (see \cite{karlsson} for details). Similarly, one can define the Floyd boundary of a locally finite graph.
\begin{defn}
	[Shortest sequence] A sequence $\{w_n\}\subset\Gamma_G$ is said to be a {\em shortest sequence} if, for all $i$, $d(e_G,w_i)=i$ and $d(w_i,w_{i+1})=1$.
\end{defn}
\begin{lemma}\label{shortest cauchy}
	If $\{w_n\}$ is a shortest sequence in $\Gamma_G$ then $\{w_n\}$ is a Cauchy sequence in $(\Gamma_G,d_f)$.
\end{lemma}
\begin{proof}
	Since $\Sigma_{r=1}^{\infty}f(r)<\infty$, given $\epsilon>0$, there exists $N\in\mathbb N$ such that $\Sigma_{r=N}^{\infty}f(r)\leq \epsilon$. As $\{w_n\}$ is shortest, for all $i,j\geq N$, $$d_f(w_i,w_j)\leq\Sigma_{r=N}^{\infty}f(r)\leq\epsilon.$$ Hence the lemma.
\end{proof}
\begin{comment}
	\begin{lemma}Given $D\geq 0$ and a Floyd function $f$, the following holds:
		
		Let $\{x_n\}$ and $\{y_n\}$ be two Cauchy sequences in $G$ such that $d(x_n,y_n)\leq D$. Then, $\{x_n\}$ and $\{y_n\}$ represent the same point in $\partial_f(G)$.
	\end{lemma}
	\begin{proof}
		
	\end{proof}

	We finish this subsection by recording the following lemma which follows from \cite[Lemma 2.5]{gerasimov-potya-floyd}.
	
	\begin{lemma}\label{qi lemma}
		Let $G$ and $H$ are two groups which are quasiisometric. Supppose there exists $D>0$ and a Floyd function $f$ such that $f(n)/f(2n)\leq D$. Then, $\partial_{f}(G)$ is homeomorphic to $\partial_f(H)$.
	\end{lemma}
	\begin{proof}
		
	\end{proof}
\end{comment}
\subsection{A tree of Cayley graphs corresponding to amalgamated free product}\label{2.3} 

Suppose $A$ and $B$ are two finitely generated groups, and $H$ is a finite group. Let $G=A\ast_H B$ be an amalgamated free product with monomorphisms $i_A: H\rightarrow A$ and $i_B: H \rightarrow B$. Choose finite generating sets $S_A$ and $S_B$ of $A$ and $B$ such that $S_A \supset i_A(H-\{1\})$ and $S_B \supset i_B(H-\{1\})$. Let $\Gamma$ be the Cayley graph of $G$ with respect to the generating set $S:=S_A \cup S_B$. We will denote the associated word metric on $\Gamma$ by $d$. We start this subsection by defining the Bass--Serre tree (\cite{serre-trees}) corresponding to $G$.

Let $\tau$ be a unit interval with vertices $v_A$ and $v_B$. We define a tree $T$, called the {\em Bass--Serre tree} of $G$, as $G\times \tau$ divided by the transitive closure of the relation $\sim$ defined as follows:  
$$(g_1,v_A)\sim (g_2,v_A) \text{ if } g_1^{-1}g_2\in A,$$ 
$$(g_1,v_B)\sim (g_2,v_B) \text{ if } g_1^{-1}g_2\in B,$$ 
$$(g_1,t) \sim (g_2,t) \text{ if } g_1^{-1}g_2\in H \text{ and } t \in \tau.$$ 
The tree $T$ comes with a natural action of $G$. For each vertex $v\in T$, we denote by $G_v$ the stabilizer of $v$ in $G$.

Let $\Gamma_A$ and $\Gamma_B$ be the Cayley graphs of $A$ and $B$ with respect to $S_A$ and $S_B$, respectively. Let $\Gamma_{i_A(H)}$ and $\Gamma_{i_B(H)}$ be the subgraphs of $\Gamma_A$ and $\Gamma_B$ having vertices $i_A(H)$ and $i_B(H)$, respectively. Since $S_A \supset i_A(H-\{1\})$ and $S_B \supset i_B(H-\{1\})$, the subgraphs $\Gamma_{i_A(H)}$ and $\Gamma_{i_B(H)}$ are subgraphs of diameter $1$. 
Define a graph $Y$ as the union of $\Gamma_A$ and $\Gamma_B$ where $i_A(h)$ is identified with $i_B(h)$ for all $h\in H$. Thus, in $Y$, we obtain a subgraph which is union of $\Gamma_{i_A(H)}$ and $\Gamma_{i_B(H)}$ where $i_A(h)$ is identified with $i_B(h)$ for all $h\in H$. We will denote this subgraph by $\Gamma_H$. Finally, we define the graph $\Gamma$ as an equivalence relation on $G\times Y$ induced by 
$$(g_1,y_1)\sim(g_2,y_2) \text{ if }y_1,y_2\in \Gamma_A \text{ and } g_2^{-1}g_1y_1=y_2,$$
$$(g_1,y_1)\sim(g_2,y_2) \text{ if }y_1,y_2\in \Gamma_B \text{ and } g_2^{-1}g_1y_1=y_2.$$
Note that $\Gamma$ is the Cayley graph of $G$ with respect to $S$. Define a map $p:\Gamma\to T$ by sending $(g,\Gamma_A)$ to $(g,v_A)$ and $(g,\Gamma_B)$ to $(g,v_B)$ for $g\in G$. As $(g,\Gamma_A)$ and $(g,\Gamma_B$) both contain the subgraph $(g, \Gamma_H)$, $p$ sends $(g,\Gamma_H)$ to $(g,v_A)$ as well as to $(g,v_B)$. Since $(g,v_A)$ and $(g,v_B)$ are $1$-distance apart in $T$, $p$ is coarsely well-defined.

\vspace{.2cm}
{\bf Notation:} We will write the edge $(g,\tau)$ of $T$ as $g\tau$. If, for $g\in G$, either $v=(g,v_A)$ or $v=(g,v_B)$ is a vertex of $T$ then the Cayley graph of $G_v$ is isometric to either $(g,\Gamma_A)$ or $(g,\Gamma_B)$. For $(g,v_A)\in T$, we denote by $\Gamma_{gA}$ the preimage of $(g,v_A)$ under $p$ which is nothing but $(g,\Gamma_A)$. Similarly, for $(g,v_B)\in T$, we denote by $\Gamma_{gB}$ the preimage of $(g,v_B)$ under $p$. For a vertex $v\in T$, we denote by $\Gamma_v$ the preimage of $v$ under $p$. We will denote the subgraphs $(g,\Gamma_H)$ by $\Gamma_{gH}$. After this point, while using the notation $g\tau,\Gamma_{gA}$ and $gA$, we will assume $g$ to be the shortest possible coset representative. Similarly, we will assume the same for cosets of $B$ and $H$.
\subsection{A tree of Cayley graphs corresponding to HNN extension} \label{phnn} Let $A$ be a finitely generated group with finite subgroups $F$ and $F'$ such that $F \simeq F'$ as groups. Let $G=A*_{F\simeq F'}$ be a HNN extension of $A$. Let $\phi:F\to F'$ be the isomorphism. Then, the presentation of $G$ is $\langle S_A,t: t^{-1}ft=\phi(f), f\in F\rangle$. Let $S_A$ be a finite generating set of $A$, which contains $F\setminus\{1\}$ and $F'\setminus\{1\}$. Let $\Gamma$ be the Cayley graph of $G$ with respect to $S:=S_A \cup \{t\}$, where $t$ is the stable letter. We will denote the word metric on $\Gamma$ by $d$. 

Firstly, we give a description of the Bass--Serre tree, say, $T$ of $G$. The vertex set of $T$, say, $V(T)$ is the set of left cosets of $A$ in $G$. The set of edges of $T$, say, $E(T)$ is the set of left cosets of $F$ in $G$. For all $g \in G$, vertices $gA$ and $gtA$ are connected by the edge $gF$.

Let $\Gamma_A$ be the Cayley graph of $A$ with respect to $S_A$. Consider the graph $Y$ induced by the following equivalence relation on $G\times \Gamma_A$ :
$$(g_1,y_1)\sim(g_2,y_2) \text{ if }y_1,y_2\in \Gamma_A \text{ and } g_2^{-1}g_1y_1=y_2.$$
Denote the equivalence class $[g,x]$ by $gx$ and $\Gamma_{gA} $ be the subgraph spanned by the vertices in $\{gx \in Y : x\in \Gamma_A \}$. These are the connected components of $Y$, which are isometric to $\Gamma_A$ with the word metric.
Now, the Cayley graph of $G$ can defined by using the graph $Y$ as follows:

For all $g \in G$, the vertices $gA$ and $gtA$ are connected in the Bass-Serre tree $T$. Similarly, join the vertices $gf \in \Gamma_{gA}$ and $gft(=gt\phi(f)) \in \Gamma_{gtA}$, for all $f \in F$.  Denote the subgraph spanned by the vertices in $\{gf,gt\phi(f): f\in F\}$ by $\Gamma_{gF}$. From this discussion, we can say that the subgraphs $\Gamma_{gA}$ and $\Gamma_{gtA}$ are connected through the subgraph $\Gamma_{gF}$. Hence, we can define the natural projection map, $p: \Gamma \to T$, such that $p(\Gamma_{gA}):=gA$, which in this case is well-defined. 

\section{A Construction of Floyd boundaries of an amalgamated free product and HNN extension}\label{3}
 {\bf Amalgamated free product case.} Let $G=A\ast_H B$ be as in Subsection \ref{2.3}. The next subsection is devoted to a construction of Floyd boundary of $G$. When $A$ and $B$ are hyperbolic and $H$ is trivial, Martin--\`{S}wi\polhk{a}tkowski gave a construction of Gromov boundary of $G$ in \cite{martin-swiat}. We closely follow that construction. Let $f$ be a Floyd function.

 {\bf Notation:}  We retain the notations used in the Subsection \ref{2.3}.  We will denote the scaled Cayley graph of $G$ with respect to the generating set $S$ by $\Gamma_f$ with the Floyd metric $d_f$. For the sake of ease, we will denote the metric subspaces $(\Gamma_{gA},d_f)$, $(\Gamma_{gB}, d_f)$ and $(\Gamma_{gH},d_f)$ by $\Gamma_{gA}^f$, $\Gamma_{gB}^f$ and $\Gamma_{gH}^f$ respectively. Similarly, when $v \in T$, we will denote the metric subspace $(\Gamma_v, d_f)$ by $\Gamma_v^f$. 
 Define $\partial_f(\Gamma_v^f):= \overline{\Gamma}_v^f- \Gamma_v^f$, where $\overline{\Gamma}_v^f$ is the metric completion of $\Gamma_v^f$. Let $\partial_f(A)$ and $\partial_f(B)$ be the Floyd boundary of $A$ and $B$, respectively.
 
 \subsection{Boundaries of the stabilizers.} Let $\delta_{Stab}(\Gamma_f)$ be the set $G\times(\partial_{f}A\sqcup\partial_{f}B)$ divided by the equivalence relation induced by
 $$(g_1,\xi_1)\sim (g_2,\xi_2) \text{ if } \xi_1,\xi_2\in \partial_{f}A, \hspace{0.07cm}g_2^{-1}g_1\in A \text{ and } g_2^{-1}g_1\xi_1=\xi_2,$$
 $$(g_1,\xi_1)\sim (g_2,\xi_2) \text{ if } \xi_1,\xi_2\in \partial_{f}B, \hspace{0.07cm} g_2^{-1}g_1\in A \text{ and } g_2^{-1}g_1\xi_1=\xi_2.$$
 
 The equivalence class of an element $(g,\xi)$ is denoted by $[g,\xi]$. The set $\delta_{Stab}(\Gamma_f)$ comes with a natural action of $G$ on the left. This also comes with a natural projection onto the set of vertices of $T$. The preimage of each vertex $v\in T$ is denoted by $\partial_{f}(G_v)$ which is identified with $\partial_f(\Gamma_v^f)$ by Lemma \ref{homeo identification}.
 
 Let $\partial T$ denotes the Gromov boundary of $T$. Then, as a set, we denote the {\em boundary} of $\Gamma_f$ as $$\delta(\Gamma_f):=\delta_{Stab}(\Gamma_f)\sqcup\partial T.$$
 
 Also, as a set, we define the {\em compactification} of $\Gamma_f$ as $$\overline{\Gamma}_f:=\Gamma_f\cup\delta(\Gamma_f).$$
 This set comes with a natural action of $G$ and with a natural map $p:\overline{\Gamma}_f\to \overline{T}$, which is coarsely well-defined. The preimage of a vertex $v\in T$ is $\Gamma_v^f\cup\partial_{f}(G_v)$ that is identified, as a set, with $\overline{\Gamma}_v^f$ which is the metric completion of $\Gamma_v^f$. 

\subsection{Two topologies on $\overline{\Gamma}_f$} \label{amalgam topo} There are two ways to put a topology on $\overline{\Gamma}_f$. Here, we give a description of these topologies, and we prove that these two are equivalent.

{\bf Topology using neighborhood system.}  For a point $x\in \Gamma_f$, we set a basis of open neighborhoods of $x$ in $\Gamma_f$ to be a basis of open neighborhoods of $x$ in $\overline{\Gamma}_f$. Now we define a basis of open neighborhoods for points of $\delta(\Gamma_f)$. Fix a vertex $v_0$ of $T$.
\begin{enumerate}
	\item Let $\xi\in\delta_{Stab}(\Gamma_f)$. Suppose $v$ is the vertex of $T$ such that $\xi\in\partial_{f}(G_v)$. Let $U$ be an open neighborhood of $\xi$ in $\overline{\Gamma}_v^f$. 
	Define $V_U$ to be the set of all $z\in\overline{\Gamma}_f\setminus\overline{\Gamma}_v^f$ such that the first edge of the geodesic $[v,p(z)]$ or the geodesic ray $[v,p(z))$ is $g \tau$ where $g$ is a vertex in $U$. Then we set $$V_U(\xi):=U\cup V_U.$$ A neighborhood basis of $\xi$ in $\overline{\Gamma}_f$ is a collection of set $V_U(\xi)$ where $U$ runs over a neighborhood basis of $\xi$ in $\overline{\Gamma}_v^f$.
	\item Let $\eta\in \partial T$. Let $T_n(\eta)$ be the subtree of $T$ consisting of those elements $x$ of $T$ for which the first $n$ edges of $[v_0,x]$ and $[v_0,\eta)$ are the same. Suppose $u_n(\eta)$ is the vertex on $[v_0,\eta)$ at the distance $n$ from $v_0$. Let $\partial (T_n{(\eta)})$ denotes the Gromov boundary of $T_n{(\eta)}$ and let $\overline{T_n(\eta)}=T_n(\eta)\cup\partial(T_n(\eta))$. We define $$V_n(\eta)=p^{-1}(\overline{T_n(\eta)}\setminus\{u_n(\eta)\}).$$ 
	We take the collection $\{V_n(\eta):n\geq 1\}$ as a basis of open neighborhoods of $\eta$ in $\overline{\Gamma}_f$.
\end{enumerate}
We skip a straightforward verification that the above collections of sets satisfy the axioms for the basis of open neighborhoods.

{\bf Metric on $\overline{\Gamma}_f.$} Recall that $d_f$ is the Floyd metric on $\Gamma_f$ and $\Gamma_{gA}^f$, $\Gamma_{gB}^f$ are the notations for the metric subspaces as indicated in the first paragraph of this section.
\begin{comment}
	
\begin{rem}\label{len}
	It is straightforward to check that $\Gamma_v^f$, for every $v \in T$, is a length space since the generating set $S$ for the Cayley graph $\Gamma$ contains $H-\{1\}$.
\end{rem}
\end{comment}
Before defining the metric on $\overline{\Gamma}_f$, we record the following observations which follows from \cite[p. 5366]{karlsson}:
\begin{lemma}\label{homeo identification}
	For $g\in G$, the natural translation map from $\Gamma_A^f$ to $\Gamma_{gA}^f$ is bi-Lipschitz and hence induces a homeomorphism from $\partial_f(A)$ to $\partial_f(\Gamma_{gA}^f)$. Similarly, $\partial_f (B)$ is homeomorphic to $\partial_f(\Gamma_{gB}^f)$. \qed
\end{lemma}
\begin{comment}
	
\begin{proof}
	Note that by our choice, $g$ is the shortest possible coset representative of $gA$. Let $d(e_G,g) = m$ and $z \in \Gamma_{gA}^f$. Let $\gamma$ be the $d-$geodesic from $e_G$ to $z$. Then $\gamma$ passes through $z'\in\Gamma_{gH}$ which implies that 
	\[
	d(e_G,z)=d(e_G,z') + d(z,z').
	\] Since for all $x, y \in \Gamma_{gH}$, $d(x,y)=1$, using triangle inequality, we have the following:
	\[
	m+ d(g,z)-2 \leq d(e_G, z) \leq m+ d(g,z)+ 2.
	\] Let $a \in A$, then $d(e_G,a)=d(g,ga)$, which implies the following inequality:
	\[
	m+d(e_G,a) -2 \leq d(e_G, ga) \leq m+d(e_G,a) +2.
	\]
	Let $x,y \in \Gamma_A^f$ and $\gamma_f$ be a path joining them in $\Gamma_A^f$ such that the length of the path in the Floyd metric is $\sum_{i=1}^{n}f(r_i)$. Then the  path $g\gamma_f$ is a path between $gx$ and $gy$ in $\Gamma_{gH}^f$. Then
	\[
	d_f(gx,gy) \leq l_f(g\gamma_f)\leq \sum_{i=1}^n f(m+r_i-2) \leq \lambda^{m-2} \sum_{i=1}^n f(r_i)=\lambda^{m-2}l_f(\gamma_f).
	\] 
	Since $d_f(x,y)=\inf\{l_f(\gamma_f)\hspace{0.15cm}| \hspace{0.15cm} \gamma_f(0)=x, \gamma_f(1)=y, \text{ and } \gamma_f \subset \Gamma_{A}^f\}$, which gives us
	\[d_f(gx,gy)\leq \lambda ^{m-2}d_f(x,y).
	\]
	Similarly, using Remark \ref{len}, we can show that 
	\[
	d_f(x,y) \leq  (1/\lambda^{m+2})d_f(gx,gy).
	\] Hence, we have the desired result.
\end{proof}
\end{comment}
Let $\eta\in\partial T$. Fix the vertex $v_0=(e_G,v_A)$ and let $\{v_n\}$ be the sequence of vertices that lies on the geodesic ray joining $v_0$ to $\eta$ in $T$. Suppose for all $i\geq 0$, $v_i$ and $v_{i+1}$ are connected by the edge $g_{i}\tau$. Choose $x_i\in \Gamma_{g_iH}$. It is straightforward to show that $\{x_i\}$ is a Cauchy sequence in $\Gamma_f$ and any other choice $\{y_i\}$ is Cauchy equivalent to $\{x_i\}.$ We call such a sequence $\{x_i\}$, a {\em Cauchy representative} of $\eta$. 

{\bf Definition of metric $\bar{d_f}$ on $\overline{\Gamma}_f$.} We define a metric $\bar{d_f}$ in the following manner:

\begin{enumerate}
	\item For $x,y\in\Gamma_f$, $\bar{d_f}(x,y):=d_f(x,y)$.
	\item Suppose $x\in \Gamma_f$ and $\xi\in \partial_f(G_v)$ for some vertex $v\in T$. Let $\{x_n\}$ be a Cauchy sequence in $\Gamma_v^f$ representing $\xi$. Then, $\bar{d_f}(x,\xi):=\lim_{n\to\infty}d_f(x,x_n)$. This limit exists as $\{x_n\}$ is Cauchy, and it is independent of the chosen sequence $\{x_n\}$.
	\item Suppose $\xi_1\in\partial_f(G_v)$ and $\xi_2\in\partial_f(G_w)$ for some $v,w\in T$. Let $\{x_n\}$ and $\{y_n\}$ be Cauchy sequences in $\Gamma_v^f$ and $\Gamma_w^f$ representing $\xi_1$ and $\xi_2$, respectively. Then, $\bar{d_f}:=\lim_{n\to\infty}d_f(x_n,y_n)$. Since $\{x_n\}$ and $\{y_n\}$ are Cauchy, this limit exists, and it is independent of chosen sequences.
	\item Let $\eta,\eta_1,\eta_2$ be three distinct points in $\partial T$ and $\xi\in\delta_{Stab}(\Gamma_f)$. Let $\{x_n\},\{y_n\},\{z_n\}$ be Cauchy representatives of $\eta,\eta_1,\eta_2$, respectively, in $\Gamma_f$. The above discussion shows that the following definition does not depend on the chosen sequences.
	\begin{enumerate}
		\item  $\bar{d}_f(\eta,z):= \lim_{n\to\infty} d_f (x_n,z)$, where $z \in \Gamma_f$.
		\item $\bar{d}_f(\eta,\xi ):= \lim_{n\to\infty} d_f(x_n, u_n)$, where $\xi \in \partial_f (G_v)$ for some $v\in T$ and $\{u_n\}$ is Cauchy sequence in $\Gamma_v^f$ representing $\xi$.
		\item $\bar{d}_f (\eta_1, \eta_2):= \lim_{n\to\infty} d_f(y_n,z_n)$.
	\end{enumerate}  
\end{enumerate}
\begin{rem}\label{restriction}
	From the definition of the metric $\bar{d}_f$, we can directly conclude that for each vertex $v\in T$, the restriction of $\bar{d_f}$ on $\overline{\Gamma}_v^f$ is the same as the completion metric of $d_f$ on $\Gamma_v^f$.
\end{rem}
\begin{prop}\label{p1}
	$\overline{\Gamma}_f$ is a metric space.
\end{prop}
\begin{proof}
	From the definition, it follows that $\bar{d}_f$ is a pseudo metric. It remains to check the positivity of $\bar{d}_f$. If $x\in\Gamma_f$ and $\xi\in\delta_{Stab}(\Gamma_f)$ then $\bar{d}_f(x,\xi)>0$. Thus, the following three cases are remaining to consider:
	
	{\bf Case 1.} Suppose $\xi_1\neq\xi_2\in\delta_{Stab}(\Gamma_f)$. If $\xi_1,\xi_2\in \partial_f(G_v)$ for some $v\in T$ then, by Remark \ref{restriction}, $\bar{d}_f(\xi_1,\xi_2)>0$. Suppose $v\neq w\in T$ such that $\xi_1\in \partial_f(G_v)$ and $\xi_2\in\partial_f(G_w)$. Let $e_1=[v,v_1]$ and $e_2=[v_n,w]$ be the first and last edge on the geodesic $[v,w]$ in $T$. Suppose $\Gamma_v^f$ and $\Gamma_{v_1}^f$ are joined through $\Gamma_{g_1H}^f$. Similarly, suppose $\Gamma_{v_n}^f$ and $\Gamma_w^f$ are joined through $\Gamma_{g_nH}^f$. Then, we have the following inequality:
	
	\begin{equation}\label{Imp}
		d_f(x,y)\geq d_f(x, \Gamma_{g_1H}^f)+ d_f(\Gamma_{g_nH}^f,y).
	\end{equation}
	Suppose $\{x_n\},$ $\{y_n\}$ are Cauchy sequences in $\Gamma_v^f$ and $\Gamma_w^f$ representing $\xi_1$ and  $\xi_2$, respectively. Then, putting $x=x_n$ and $y=y_n$ in Equation \ref{Imp} and applying limit on both sides, we get
	\[
	\bar{d}_f(\xi_1,\xi_2)\geq\bar{d}_f(\xi_1, \Gamma_{g_1H}^f)+ \bar{d}_f(\Gamma_{g_nH}^f,\xi_2).
	\] 
	By Remark \ref{restriction} and the fact that $\partial (\Gamma_v^f)$ and $\Gamma_{g_1H}^f$ are compact subset of $\overline{\Gamma}_v^f$, we have
	\[
	\bar{d}_f(\xi_1, g_1) \geq \bar{d}_f(\Gamma_{g_1H}^f, \partial (\Gamma_v^f)) >0.
	\] Hence, we conclude that $\bar{d}_f(\xi_1,\xi_2) >0$.
	
	{\bf Case 2.} Suppose that $\xi \in \delta_{Stab}(\Gamma_f)$, $\eta \in \partial T$. Let $v\in T$ such that $\xi\in\partial_f(G_v)$ and let $w_n$ be a Cauchy representative of $\eta$. Let $e_1$ be the first edge on the geodesic ray $[v,\eta)$. Suppose $\Gamma_v^f$ and $\Gamma_{v_1}^f$ are joined through $\Gamma_{g_1H}^f$. Then, as in Case 1, we get 
	\[
	d_f(x, w_n) \geq d_f(x, \Gamma_{g_1H}^f) \hspace{0.25cm} \text{ for all } x\in \Gamma_v \text{ and sufficiently large } n.
	\] 
	Let $\{x_n\}$ be a Cauchy representative of $\xi$ in $\Gamma_v^f$. Putting $x=x_n$ in the last equation and taking the limit on both sides, we get
	\[
	\bar{d}_f(\xi, \eta) \geq \bar{d}_f(\xi,\Gamma_{g_1H}^f)>0.
	\]

	{\bf Case 3.} Suppose $\eta_1, \eta_2 \in \partial T$ and $\{x_n\}$ and $\{y_n\}$ be their respective Cauchy representatives. Let $\gamma_1$ and $\gamma_2$ be the geodesics from $v_A$ to $\eta_1$ and $\eta_2$, respectively. Then there exists a unique vertex $v_0 \in T$ such that 
	$\gamma_1 \cap \gamma_2 = [v_A, v_0].$ Let $e_1=[v_0,v_1]$ and $e_2=[v_0,v_2]$ be the first edges on $\gamma_1$ and $\gamma_2$ just after $[v_A,v_0]$. Suppose $\Gamma_{v_0}^f$, $\Gamma_{v_1}^f$ are joined through $\Gamma_{g_1H}^f$ and $\Gamma_{v_0}^f$, $\Gamma_{v_2}^f$ are joined through $\Gamma_{g_2H}^f$. Then, for all sufficiently large $n$, each geodesic joining $x_n$ and $y_n$ in $\Gamma_f$ passes through $\Gamma_{g_1H}^f$ and $\Gamma_{g_2H}^f$ which implies that
	$d_f(x_n,y_n) \geq d_f(\Gamma_{g_1H}^f, \Gamma_{g_2H}^f) >0.$ Hence, we conclude that $\bar{d}_f(\eta_1,\eta_2)>0.$
\end{proof}
Now, we are ready to prove the following:
\begin{prop}\label{p2}
	The Floyd boundary $\partial_f(G)$ of $G$ is isometric to $\delta(\Gamma_f)$.
\end{prop}
\begin{proof}
	Recall that for every $z \in \Gamma$, $p(z)$ has two possible. For this proof, we will choose the value closest to $v_A$.

	Let $i: (\Gamma_f,d_f) \rightarrow (\overline{\Gamma}_f,\bar{d}_f)$ be the natural inclusion map, which is an isometric embedding. Hence it extends to an isometric embedding $\bar{i}:\Gamma_f\cup\partial_f(G)\to\overline{\Gamma_f\cup\delta(\Gamma_f)}$, where $\overline{\Gamma_f\cup\delta(\Gamma_f)}$ denotes the Cauchy completion of $\Gamma_f\cup\delta(\Gamma_f)$. We show that $\bar{i}$ is in fact an isometry. Note that $\bar{i}(\Gamma_f \cup \partial_f (G)) = \overline{i(\Gamma_f)}$ (which is nothing but the Cauchy completion of $\Gamma_f$).
	Let $\{w_n\} \in \Gamma$ be a shortest sequence. Then, by Lemma \ref{shortest cauchy}, $\{w_n\}$ is a Cauchy sequence in $(\Gamma_f,d_f)$. Also, the distance between 
	$v_A$ and $p(w_n)$ is non-decreasing. Hence, we have the following two cases:
	
	{\bf Case 1.} There exists $N \in \mathbb{N}$ such that for all sufficiently large n, the distance in $T$ between $v_A$ and $p(w_n)$ is $N$. Since $\{w_n\}$ is a shortest sequence in $\Gamma$, there exists a vertex $v_N\in T$ such that $d_T(v_A,v_N)=N$ and $w_n \in \Gamma_{v_{N}}^f$ for all large $n$. Hence, $\{w_n\}$ converges to a point $\xi \in \partial_f(\Gamma_{v_N}^f)\subset\delta(\Gamma_f)$. 
	
	{\bf Case 2.} The distance between $v_A$ and $p(w_n)$ goes to $\infty$ as $n\rightarrow \infty.$ Since $\{w_n\}$ is shortest, there exists a geodesic ray, say, $\gamma$ in $T$ such that $\gamma(0)=v_A$ and $p(w_n)\in \gamma$. Let $\eta\in\partial T$ be the point represented by $\gamma$. Then, a subsequence $\{w_{n_k}\}$ becomes a Cauchy representative of $\eta$. Hence, $\{w_n\}$ converges to $\eta$ in $(\overline{\Gamma}_f,\bar{d}_f)$. 
	
	Let $z \in \partial_f (G)$. Then there exists a shortest sequence $\{w_n\} \in \Gamma_f$ such that $w_n \rightarrow z$ in $\Gamma_f \cup \partial_f (G)$. From the above two cases, we know that there exists $z' \in \delta(\Gamma_f)$ such that $w_n \rightarrow z'$ in $\Gamma_f \cup \delta (\Gamma_f)$. Hence, by continuity of $\bar{i}$, we have that $\bar{i}(z)=z'$. This implies that $\overline{i}(\partial_f (G)) \subset \delta (\Gamma_f)$. Hence, we conclude that 
	$\overline{i(\Gamma_f)}= \bar{i}(\Gamma_f\cup\partial_f(G)) \subset \Gamma_f \cup \delta \Gamma_f.$ Since, by construction, $i(\Gamma_f)$ is dense in $\Gamma_f \cup \delta \Gamma_f$, we have that $\overline{\Gamma_f \cup \delta \Gamma_f} =\overline{i(\Gamma_f)} \subset  \Gamma_f \cup \delta \Gamma_f$. This in turn implies that $\overline{\Gamma_f \cup \delta \Gamma_f}=\Gamma_f \cup \delta \Gamma_f=\overline{i(\Gamma_f)}$. Hence, $\bar{i}$ is an isometry between $\Gamma_f \cup \partial_f G$ and $\Gamma_f \cup \delta (\Gamma_f),$ that gives us the desired result.
\end{proof}

\subsection{Equivalence of topologies on $\overline{\Gamma}_f$} Let $\mathcal T_o$ be the topology defined by the neighborhood basis on $\overline{\Gamma}_f$ and let $\mathcal T_m$ be the topology on $\overline{\Gamma}_f$ induced by the metric $\bar{d}_f$. In this subsection, we show that $\mathcal T_o$ is the same as $\mathcal T_m$.
\begin{prop}\label{eq1}
	$\mathcal T_o=\mathcal{T}_m$.
\end{prop}
\begin{proof}
	{\bf $\mathcal T_o$ is finer than $\mathcal T_m$:} Let $\xi\in\delta_{Stab}(\Gamma_f)$. Let $v\in T$ represents coset $g_0A$ with $g_0$ being a reduced representative such that $\xi\in\partial_f(G_v)$. Let $B_r$ be a $r$-radius ball about $\xi$ in $(\overline{\Gamma}_f, \bar{d}_f)$. We show that there exists a neighborhood $U$ of $\xi$ in $\overline{\Gamma}_v^f$ such that $V_U(\xi)\subset B_r(\xi)$. Choose $N \in \mathbb{N}$ such that 
	$\sum_{j=N}^\infty f(j) < \frac{r}{3}$ and $d(e_G, \Gamma_{g_0A})< N-1$. Let 
	$$W:= \{x \in \Gamma \hspace{0.15cm} | \hspace{0.15cm} d(x, e_G) \leq N\}.$$ Note that $W$ is a compact subset of $(\overline{\Gamma}_f,\bar{d}_f) $.
	Define $U:= (B_{r/3}-W) \cap \Gamma_v^f$. It is evident that $U$ is an open neighborhood of $\xi$ in $\overline{\Gamma}_v^f.$ 
	Let $z $ be an element in $ V_{U}(\xi)$ and let $e_1=[v,v_1]$ be the first edge either on the geodesic $[v,p(z)]$ or on the geodesic ray $[v,p(z))$. 
	Assume that $\Gamma_v^f$ and $\Gamma_{v_1}^f$ are joined through $\Gamma_{gH}^f$. 
	Since every path in $\Gamma_f$ joining $e_G$ and $z$ passes through $\Gamma_{gH}^f$, we have
	\begin{equation}\label{Imp1}
		\bar{d}_f(e,\Gamma_{gH}^f)+\bar{d}_f(\Gamma_{gH}^f, z)	\leq	\bar{d}_f(e,z). 
	\end{equation}
	Since a representative $g_1$ (not necessarily reduced) of $gH$ is contained in $U$, we get $d(e_G,\Gamma_{gH})\geq N$. Thus, we have
	$\bar{d}_f(\Gamma_{gH}^f,e) \geq \sum_{i=0}^{N-1} f(i), $ and $  \bar{d}_f(e,z) \leq \sum_{i=0}^\infty f(i).$
	By using inequality \ref{Imp1}, we get
	$\bar{d}_f (\Gamma_{gH}^f,z) \leq \sum_{i=N}^\infty f(i)< \frac{r}{3}.$
	By using triangle inequality, we get
	\[
	\bar{d}_f(\xi, z) \leq \bar{d}_f(\xi, \Gamma_{gH}^f) + \bar{d}_f(\Gamma_{gH}^f,z) + f(N) <  \frac{r}{3}+\frac{r}{3}+\frac{r}{3}=r.\]
	Hence, we conclude that $B_r \supset V_U(\xi).$
	
	Now, let $\eta \in \partial T$ and let $U_{s}$ be an open ball of radius $s\geq 0$ about $\eta$ in $(\overline{\Gamma}_f, \bar{d}_f)$. Choose $N$ such that $\Sigma_{j=N}^{\infty}f(j)<\frac{s}{3}$. We have a unique sequence $v_i$ in $T$, with $v_0=v_A$, which along the unique geodesic converges to $\eta$. Let $g_i\tau$ be the edge between $v_i$ and $v_{i+1}$ with $g_i$ being the reduced representative of the coset associated to vertex $v_{i+1}$. Then $\{g_i\}$ is a Cauchy representative of $\eta$. Thus, $\{g_i\}$ converges to $\eta$. Choose $N_0 >N$ such that $
	\bar{d}_f(\Gamma_{g_{N_0}H}^f,\eta) < \frac{s}{3}.$  It is easy to see that 
	$d(\Gamma_{g_{N_0}H}, e_G) \geq N_0.$ For $z \in V_{N_0}(\eta)$, every path between $e_G$ and $z$ passes through $\Gamma_{g_{N_0}H}$.
	Hence, by using inequality \ref{Imp1} and triangle inequality, we get
	\[
	\bar{d}_f(z, \eta) \leq \bar{d}_f(z,\Gamma_{g_{N_0}H}^f)+  \bar{d}_f(\eta,\Gamma_{g_{N_0}H}^f) + f(N_0) <  \frac{s}{3}+ \frac{s}{3} + \frac{s}{3} =r. \]
	Hence $V_{N_0}(\eta)\subset U_s$.
	
	{\bf $\mathcal T_m$ is finer than $\mathcal T_o$:} Let $\xi\in\partial_{f}(\Gamma_v^f)\subset \delta_{Stab}({\Gamma_f})$ for some vertex $v\in T$. Let $U$ be a neighborhood of $\xi$ in $\overline{\Gamma}_v^f$ such that $V_U(\xi)$ is a neighborhood of $\xi$ in $\overline{\Gamma}_f$. Since $U$ is open in $\overline{\Gamma}_v^f$, there exists $r\geq 0$ such that the $\overline{\Gamma}_v^f$-ball of radius $r$ contained in $U$. Let $B_r(\xi)$ is the $r$-radius ball in $(\overline{\Gamma}_f,\bar{d}_f)$. Then, it follows from the definition of neighborhoods that $B_r(\xi)\subset V_U(\xi)$. Let $\eta\in \partial T$ and let $V_M(\eta)$ be a neighborhood in $\overline{\Gamma}_f$. Choose $t\geq 0$ such that $\Sigma_{j=M}^{\infty}f(j)<\frac{t}{3}$. Let $B_t(\eta)$ is ball of radius $t$ in $(\overline{\Gamma}_f,\bar{d}_f)$. Then, again, it is easy to see that $V_M(\eta)\subset B_t(\eta)$. 
\end{proof}
\subsection{HNN extension case}\label{hnn construnction}
In this subsection, we retain the notations of Subsection \ref{phnn}. Let $f$ be a Floyd function. We will denote the scaled Cayley graph of $G$ with respect to the generating set $S$ by $\Gamma_f$ with the Floyd metric $d_f$. For the sake of ease, we will denote the metric subspaces $(\Gamma_{gA},d_f)$ by $\Gamma_{gA}^f$. Similarly, when $v \in T$, we will denote the metric subspace $(\Gamma_v, d_f)$ by $\Gamma_v^f$.  Define $\partial_f(\Gamma_v):= \overline{\Gamma}_v^f- \Gamma_v^f$, where $\overline{\Gamma}_v^f$ is the metric completion of $\Gamma_v^f$.
\vspace{.2cm}

{\bf Boundaries of the stabilizers.} Let $\delta_{Stab}(\Gamma_f)$ be the set $G\times\partial_{f}(A)$ divided by the equivalence relation induced by
$$(g_1,\xi_1)\sim (g_2,\xi_2) \text{ if } \xi_1,\xi_2\in \partial_{f}A, \hspace{0.07cm}g_2^{-1}g_1\in A \text{ and } g_2^{-1}g_1\xi_1=\xi_2.$$

The equivalence class of an element $(g,\xi)$ is denoted by $[g,\xi]$. The set $\delta_{Stab}(\Gamma_f)$ comes with a natural action of $G$ on the left. This also comes with a natural projection onto the set of vertices of $T$. A result similar to Lemma \ref{homeo identification} holds for $G$, which implies that $\partial_f (A)$ is homeomorphic to $\partial_f(\Gamma_{gA}^f)$ for all $g\in G$. Hence, for every vertex $v \in T$, $\partial_f(G_v)$ can be identified with $\partial_f(\Gamma_v^f)$.

Let $\partial T$ denotes the Gromov boundary of $T$. Then, as a set, we define the {\em boundary} of $\Gamma$ as $$\delta(\Gamma_f):=\delta_{Stab}(\Gamma_f)\sqcup\partial T.$$

Also, as a set, we define the {\em compactification} of $\Gamma_f$ as $$\overline{\Gamma}_f:=\Gamma_f\cup\delta(\Gamma_f).$$
This set comes with a natural action of $G$ and with a natural map $p:\overline{\Gamma}_f\to \overline{T}$, which in this case is well-defined. The preimage of a vertex $v\in T$ is $\Gamma_v^f\cup\partial_{f}(G_v)$ that is identified, as a set, with $\overline{\Gamma}_v^f$ which is the metric completion of $\Gamma_v^f$.
As in Subsection \ref{amalgam topo}, we can similarly define topology $\mathcal{T}_o$ and metric $\bar{d}_f$ on $\Gamma_f$. Following the steps employed in previous subsections, one can easily show that $\overline{\Gamma}_f$ with the  topology $\mathcal{T}_o$ and the topology induced by the metric $\bar{d}_f$ are homeomorphic as well as $\delta(\Gamma_f)$ with the metric $\bar{d}_f$ is isometric to $\partial_f(G)$. Hence, $\delta(\Gamma_f)$ with the subspace topology induced by $\mathcal{T}_o$ is homeomorphic to $\partial_f(G)$.
\section{Homeomorphism types of Floyd boundaries of free products}\label{4}
The main goal of this section is to prove the following theorem. In this section, assume all the groups to be infinite unless mentioned otherwise.

\begin{theorem}\label{homeo free product}
	Suppose $G_1=A_1\ast B_1$ and $G_2=A_2\ast B_2$ are two free products of infinite groups. Let $f_1$ and $f_2$ are two Floyd functions such that $\partial_{f_1}(A_1)\simeq\partial_{f_2}(A_2)$ and $\partial_{f_1}(B_1)\simeq \partial_{f_2}(B_2)$ then there is a homeomorphism from $\partial_{f_1}(G_1)$ to $\partial_{f_2}(G_2)$.
\end{theorem}
An easy application of induction gives the following:
\begin{cor}\label{gen free product}
	For $k\geq2$, suppose $G_1=A_1\ast\dots\ast A_k$ and $G_2=B_1\ast\dots\ast B_k$ are two free products of infinite groups. Let $f_1$ and $f_2$ are two Floyd functions such that $\partial_{f_1}(A_i)\simeq\partial_{f_2}(B_i)$ for $1\leq i\leq k$. Then, $\partial_{f_1}(G_1)\simeq\partial_{f_2}(G_2)$. \qed
\end{cor}
{\bf Notation:} Let $h_A:\partial_{f_1}(A_1)\to\partial_{f_2}(A_2)$ and $h_B:\partial_{f_1}(B_1)\to\partial_{f_2}(B_2)$ be the fixed homeomorphism of Floyd boundaries. We denote by $q:\partial_{f_1}(A_1)\cup\partial_{f_1}(B_1)\to\partial_{f_2}(A_2)\cup\partial_{f_2}(B_2)$ the homeomorphism induced by $h_A$ and $h_B$. Let $T_1$ and $T_2$ be the Bass--Serre trees of $A_1\ast B_1$ and $A_2\ast B_2$ respectively. Finally, we denote by $\Gamma_{f_1}$ and $\Gamma_{f_2}$ the Cayley graphs of $G_1$ and $G_2$, respectively. 
\begin{rem}
	The construction of a Cayley graph of a free product is same as that of $\Gamma_f$ as in Section \ref{2.3} by taking $H$ to be trivial.
\end{rem}

Suppose $G$ and $H$ are two finitely generated groups, and $f$, $f'$ are two Floyd functions for $G$ and $H$, respectively. Let $\Gamma_G$ and $\Gamma_H$ be the Cayley graphs of $G$ and $H$ respectively with the Floyd metrics $d_{f,G}$ and $d_{f',H}$. Then, $\overline{\Gamma}_G=\Gamma_G\cup\partial_{f}(G),\overline{\Gamma}_H=\Gamma_H\cup\partial_{f'}(H)$ are compact metrizable spaces. Then one has metric subspaces  $G\cup\partial_{f}(G)$ and $H\cup\partial_{f'}(H)$ with metric $d_{f,G}$ and $d_{f',H}$, respectively.. The following lemma is an analog of \cite[Lemma 4.2]{martin-swiat} in the context of Floyd boundaries.

\begin{lemma}\label{main lemma}
	Let $\psi:\partial_{f}(G)\to\partial_{f'}(H)$ be a homeomorphism. Then there is a bijection $b:G\to H$ such that $b(1)=1$ and $b\cup \psi:G\cup\partial_{f}(G)\to H\cup\partial_{f'}(H)$ is a homeomorphism.
\end{lemma}

\begin{proof}
	Define a map $\pi:G\to \partial_f(G)$ which takes $g\in G$ to $\xi\in \partial_f(G)$ such that $d_{f,G}(g,\partial_f(G))=d_{f,G}(g,\xi)$ (such a $\xi$ exists as $\partial_{f}(G)$ is compact). 
	An easy application of triangle inequality shows that $\pi(G)$ is dense in $\partial_f(G)$. Order the elements of $G\setminus\{1\}$ and $H\setminus\{1\}$ into sequences $\{g_k\}_{k\in\mathbb N}$ and $\{h_k\}_{k\in\mathbb{N}}.$ Define $b(1)=1$. To choose the required $b$, iterate the following two steps alternatively.
	
	{\bf Step 1.} Suppose $k$ is the smallest number for which $b(g_k)$ is not yet defined. Since $H$ is dense in $H\cup\partial_{f'}(H)$, choose some $l\in\mathbb N$ such that $h_l$ is not an image of any $g_i$ under the map $b$ and $$d_{f',H}(h_l,\psi(\pi(g_k)))<\dfrac{1}{k}.$$ Then, define $b(g_k)=h_l$.
	
	{\bf Step 2.} Suppose $k$ is the smallest number for which $h_k$ is not chosen as the image of any $g\in G$ under $b$. Since $\pi(G)$ is dense in $\partial_{f}(G)$, hence $\psi(\pi(G))$ is dense in $\partial_{f'}(H)$. Choose $g\in G\setminus\{1\}$ such that $b$ has not yet been defined on $g$ and $$d_{f',H}(h_k,\psi(\pi(g)))<d_{f',H}(h_k,\partial_{f'}(H))+\dfrac{1}{k}.$$ Then, define $b(g)=h_k$.
	
	By performing the above two steps alternatively, we see that $b$ is a bijection. Since $G \cup \partial_{f}(G)$ and $H \cup\partial_{f'}(H)$ are compact metrizable, to prove that $b\cup \psi$ is a homeomorphism, it is sufficient to prove that $b\cup \psi$ is continuous. 
	
	Consider a sequence $\{x_i\}$ of $G$ that converges to $\xi\in\partial_{f}(G)$. Then, $\lim_{i}d_{f,G}(x_i,\xi)=0$ and thus $\lim_{i}d_{f,G}(x_i,\partial_f(G))\to 0$. This implies that $\lim_{i}d_{f,G}(x_i,\pi(x_i))=0$ and , by using triangle inequality, therefore $\pi(x_i)$ converges to $\xi$. Hence, by continuity of $\psi$, $\psi(\pi(x_i))$ converges to $\psi(\xi)$. From the definition of $b$, it follows that $\lim_{i}d_{f',H}(b(x_i),\psi(\pi(x_i)))=0$. Thus, $b(x_i)$ converges to $\psi(\pi(\xi))$. Hence, $b\cup \psi$ is continuous.
\end{proof}

We immediately have the following:
\begin{cor}\label{main cor}
	There exists a bijection $b:G\to H$ such that the following holds:
	
	Let $\xi\in \partial_{f'}(H)$ and let $U_2$ be an open neighborhood of $\xi$ in $\overline{\Gamma}_H$. Then there exists a neighborhood $U_1$ of $f^{-1}(\xi)$ in $\overline{\Gamma}_G$ such that $b(G\cap U_1)\subset U_2$.
\end{cor}
\begin{proof}
Let $b$ be the bijection as in Lemma \ref{main lemma}. Thus, $\bar{b}:=b\cup f:G\cup\partial_{f}(G)\to H\cup\partial_{f'}(H)$ is a homeomorphism. We prove that $b$ is the required bijection. Since $U_2$ is an open neighborhood of $\xi\in \partial_{f'}(H)$, viewing $H\cup\partial_f'(H)$ as a subspace of $\overline{\Gamma}_H$, $\bar{b}^{-1}(U_2)$ is an open neighborhood of $f^{-1}(\xi)$ in $G\cup\partial_{f}(G)$. Since $K=(G\cup\partial_f(G))\setminus\bar{b}^{-1}(U_2)$ is compact and $f^{-1}(\xi)\in \overline{\Gamma}_G\setminus K$, there exists an open neighborhood $U_1$ of $f^{-1}(\xi)$ in $\overline{\Gamma}_G$ such that $U_1\subset\overline{\Gamma}_G\setminus K$. Now it is clear that $b(G\cap U_1)\subset U_2$.
\end{proof}

Using the arguments in the Lemma \ref{main lemma}, we get the following topological result as a corollary, which we will use in the coming sections.
\begin{cor}\label{topo1}
	Let $M,$ $N$ be compact metric spaces with countable dense subspaces $X \subset M$, $Y \subset N$ such that every element of $X$ and $Y$ are isolated in $M$ and $N,$ respectively. Denote $\partial M:= M - X$ and $\partial N:= N-Y$. Assume that there exists a homeomorphism $\psi: \partial M \rightarrow \partial N$. Then there exists a bijection $b: X \rightarrow Y$ such that the map $b \cup \psi: M \rightarrow N$ is a homeomorphism. \qed
\end{cor}
Using Corollary \ref{topo1}, we prove the following topological lemma, which turns out to be useful later.
In the following lemma, the closure of a subset $U$ of a metric space $X$ shall be denoted by $\overline{U}$.

\begin{lemma}\label{main topo lemma}	
	Let $(X,d)$ be a compact metric space and $Y$ be a countable dense subset of $X$ containing all isolated points of $X$. Define $\partial X:= X\setminus Y.$ Then, there exist metric subspaces $Y_1$, $Y_2$ of $Y$ such that
	\begin{enumerate}
		\item $Y= Y_1 \sqcup Y_2$.
		\item There exist bijections $\varphi_1: Y \rightarrow Y_1$ and $\varphi_2 : Y \rightarrow  Y_2$, which induce homeomorphisms $\bar{\varphi_1}: \overline{Y}=X \rightarrow Y_1 \cup  \partial X$ and $\bar{\varphi_2} : \overline{Y}=X \rightarrow Y_2 \cup \partial X$ such that $\bar{\varphi_i} \bigr|_{\partial X}$ is the identity map.
	\end{enumerate}

\end{lemma}

\begin{proof}
	Let $(X', d')$ be a metric space such that $\psi: X \rightarrow X'$ is an isometric embedding. Let $Y'=\psi(Y)$ and let $\partial X'=X'\setminus Y'=\psi(\partial X)$. Let $q:X\sqcup X'\to (X\sqcup X')/_{\sim}$ be a quotient map where $x\sim \psi(x)$ for all $x\in\partial X$. Note that , as a set, $(X\sqcup X')/_{\sim}$ is same as $Z:=\partial X\sqcup Y\sqcup Y'$. Now, by \cite[Lemma 5.24, I.5]{bridson-haefliger}, we define a metric $d_q$ on $Z$ as follows: 
	\begin{itemize}
		\item $d_q(x,y):= d(x,y)$ if $x,y \in Y \cup \partial X$ 
		\item $d_q(x,y):= d'(x,y)$ if $x,y \in Y'$.
		\item $d_q(x,y):= d'(x, \psi(y))$ if $x \in Y'$, $y \in \partial X$.
		\item $d_q(x,y):= \inf_{z \in \partial X} \left\{d(x, z ) + d'(y, \psi(z))\right\}$ if $x \in Y$, $y \in Y'$.
	\end{itemize}
	The metric $d_q$ is compatible with the topology defined on $Z$ by the quotient map $q$. It is clear from the definition of $d_q$ that the restriction of $q$ on either $X$ or on $X'$ is an isometric embedding. Also, $Y \sqcup Y'$ is a countable dense subset of $Z$, which contains all isolated points of $Z$. 
	Now, we define
	$$
	\partial Z:=Z\setminus (Y \sqcup Y').$$
	Note that the restriction of $q$ on $\partial X$ is an isometry onto $\partial Z$, which is nothing but the identity map. Since $Z$ is compact, by applying Corollary \ref{topo1},
	we get a bijection
	$\varphi: Y \sqcup Y' \rightarrow Y$ and it induces a homeomorphim $\bar{\varphi}: Z \rightarrow X.$ Let $\varphi(Y)=Y_1$ and $\varphi(Y')=Y_2$. Then, $\bar{\varphi}(\overline{Y})=Y_1\cup\partial X$ and $\bar{\varphi}(Y'\cup\partial X)=Y_2\cup\partial X$.
	We take $\varphi_1$ to be the restriction of $\varphi$ to $Y$ and $\varphi_2$ to be the composition $\varphi\bigr|_{Y'}\circ q \circ \psi$ restricted to $Y$. 
	Hence the lemma.
\end{proof}
\begin{rem}\label{rem1}
	(1) The previous lemma can be generalized to $n$ many $Y_i$'s with the properties as in the lemma.
	
	(2) Using the notations used in Lemma \ref{main topo lemma}, assume there $n$ many $Y_i$'s with the properties mentioned in the lemma. Let $\xi \in \partial X$ and $U_2$ be a neighborhood of $x$ in $X$. It is straightforward to check that there exists a neighborhood $U_1$ of $x$ in $X$ such that $\bar{\varphi}_i(U_1) \subset U_2$.
\end{rem}

{\bf Isomorphism between $T_1$ and $T_2$.} Corresponding to homeomorphisms $h_A$ and $h_B$, let $\alpha$ and $\beta$ be the bijections by Lemma \ref{main lemma}. For $i=1,2$, let $\tau_i=[v_i,u_i]$ be the edge of $T_i$ such that $v_i$ and $u_i$ are stabilized by $A_i$ and $B_i$ respectively. Recall that each non-trivial element $g\in G_1$ can be expressed uniquely, in {\em reduced form}, as $g=a_1b_1,\dots ,b_n$, with $a_j\in A_1\setminus\{1\},b_j\in B_1\setminus\{1\}$, allowing also that $a_1=1$ and that $b_n=1$. Define a map $\iota:T_1\to T_2$ that maps any edge $a_1b_1\dots a_nb_n\tau_1$ to the edge $\alpha(a_1)\beta(b_1)\dots \alpha(a_n)\beta(b_n)\tau_2$. It is easy to check that $\iota$ is an isomorphism such that $\iota(\tau_1)=\tau_2$.

Now, we are ready to prove that $\partial_{f_1}(G_1)$ is homeomorphic to $\partial_{f_2}(G_2)$. For that it is sufficient to prove that $\delta(\Gamma_{f_1})$ is homeomorphic to $\delta(\Gamma_{f_2})$. Note that, if $\xi\in\partial_{f_1}(A_1)$ and $g\in G_1$ then the set of all representatives of $[g,\xi]$ are of the form $(ga,a^{-1}\xi)$ where $a\in A_1$. Similarly, if $\xi\in\partial_{f_1}(B_1)$ and $g\in G$ then the set of all representatives are of the form $(gb,b^{-1}\xi)$ for $b\in B_1$. When $\xi\in\partial_{f_1}(A_1)$, we choose a unique $ga=a_1b_1,\dots ,a_nb_n$ for which $n$ is smallest possible (in this case we have $b_n\neq 1$). When $\xi\in \partial_{f_1}(B_1)$, we choose a unique $gb=a_1b_1,\dots ,a_nb_n$ for which $b_n=1$. These representatives of an element $[g,\xi]\in\delta_{Stab}(\Gamma_{f_1})$ are called {\em reduced represenatives.} Finally, we define a map $F:\delta(\Gamma_{f_1})\to\delta(\Gamma_{f_2})$ in the following manner:
\begin{enumerate}
	\item Let $[g,\xi]\in\delta_{Stab}(\Gamma_{f_1})$ and let $(g,\xi)$ be its reduced representative. Define $$F([a_1b_1,\dots ,a_nb_n,\xi])=[\alpha(a_1)\beta(b_1),\dots ,\alpha(a_n)\beta(b_n),q(\xi)].$$
	\item Let $\eta\in \partial T_1$. We can represent it as an infinite word $\eta=a_1b_1,\dots $ such that for each $n$, the subword consisting of its first $n$ letters corresponds to the $n$-th edge of the geodesic from $v_1$ to $\eta$ via the correspondence $g\to g.\tau_1$. Define $$F((a_1b_1,\dots ))=\alpha(a_1)\beta(b_1),\dots $$
	where the infinite word on the right gives a geodesic ray in $T_2$ starting from $v_2$.
\end{enumerate}

Note that the restriction of $F$ to $\partial(T_1)$ is same as the map $\partial(T_1)\to\partial(T_2)$ induced from the isomorphism $\iota:T_1\to T_2$.
\begin{lemma}\label{hom}
	The map $F$ is a homeomorphism.
\end{lemma}
\begin{proof}
	From the definition of $F$, it follows that $F$ is a bijection. To show that $F$ is a homeomorphism, it is sufficient to prove that $F$ is continuous. There are two cases to be considered:
	
	{\bf Case 1.} Let $\xi\in \delta_{Stab}(\Gamma_{f_1})$ and let $v$ be the vertex of $T_1$ such that $\xi\in \partial_{f_1}(G_v)$. Let $U_2$ be an open neighborhood of $F(\xi)$ in $\overline{\Gamma}_{\iota{(v)}}^{f_2}$. By Corollary \ref{main cor}, we have an open neighborhood $U_1\subset\overline{\Gamma}_{v}^{f_1}$ of $\xi$. Then, from the definition of neighborhoods in $\overline{\Gamma}_{f_1}$ and in $\overline{\Gamma}_{f_2}$, it follows that $F(V_{U_1}\cap\delta(\Gamma_{f_1}))\subset V_{U_2}(F(\xi))\cap\delta(\Gamma_{f_2})$.
	
	{\bf Case 2.} Let $\eta\in\partial (T_1)$. For an integer $n\geq 1$, consider the subtree $(T_2)_n(F(\eta))\subset T_2$, defined with respect to $v_2$. Let $(T_1)_n(\eta)=\iota^{-1}((T_2)_n(F(\eta)))$, which is subtree of $T_1$ with respect to the base vertex $v_1$. Again, from definition of neighborhoods, it follows that $F(V_n(\eta)\cap\delta(\Gamma_{f_1}))=V_n(F(\eta))\cap\delta(\Gamma_{f_2})$.
\end{proof}
This completes the proof of Theorem \ref{homeo free product}.
We end this section by proving the following interesting case. Let $f$ be  a Floyd function.
\begin{prop}\label{finite amalgam}
	Let $G= A*A, K=A*F$, where $A$ is an infinite group and $F$ is a finite group. Then $\partial_f (G)$ is homeomorphic to $\partial_f (K)$.
\end{prop}
\begin{proof}
	Assume that $F=\{e,h_1,h_2,\dots h_{n}\}$. From Lemma \ref{main topo lemma}, we have $A_i$ such that $A= \sqcup_{i=1}^n A_i$, where $1\in A_1$ and, for $1\leq i\leq n$, there exists a bijection $\varphi_i: A \rightarrow A_i$ which extends to a homeomorphism $\bar{\varphi}_i: \overline{A} \rightarrow 
	\overline{A}_i$ such that $\varphi_1(1)=1$. For the sake of distinction, let $G=A*A'$ where $A'=A$. Let $\varphi_{i}': A \rightarrow A_i'$ which extends to homeomorphisms $\bar{\varphi'}_{i}: \overline{A} \rightarrow \overline{A'}_i$.
	We denote the scaled Cayley graphs of $K$ and $G$ by $\Gamma_f$ and $\Gamma'_f$ and Bass--Serre trees by $T_1$ and $T_2$, respectively. Let $V$ be the set of vertices of $T_1$ that represents the cosets of $A$ and $V(T_2)$ is the vertex set of $T_2$. From discussions in Section \ref{3}, we know that it is sufficient to produce a homeomorphism between $\delta(\Gamma_f)$ and $\delta(\Gamma_f')$.
	
	Since $F$ is a finite group, $\partial_f (F)$ is empty. This implies that for every $v \in T_1$ representing a coset of $F$, $\partial_f (\Gamma_v^f)$ is empty. Progressing along the lines of the proof of Theorem \ref{homeo free product}, we will define a homeomorphism $\phi: V \cup \partial T_1 \rightarrow V(T_2)\cup \partial T_2$, where the domain and the codomain have the subspace topology induced from $\overline{T}_1$ and $\overline{T}_2$. Let $v \in V$ represents $gA$. Then $g$ has a unique reduced form as discussed in Section \ref{4} such that  $g=a_1h_{i_1}a_2h_{i_2}\dots  a_{n}h_{i_k}$. Let $m$ denotes a natural number.
	Define a map $\iota: V \rightarrow V(T_2)$ as follows:  
	\[\iota(gA):= \begin{cases}
		\varphi_{i_1}(a_1)\varphi'_{i_2}(a_2)\dots\varphi'_{i_{2m-2}}(a_{2m-2}) \varphi_{i_{2m-1}}(a_{2m-1})A' & \text{if } g\notin A, \hspace{0.15cm} k=2m-1, \\ 
		\varphi_{i_1}(a_1)\varphi'_{i_2}(a_2)\dots \varphi_{i_{2m-1}}(a_{2m-1})\varphi'_{i_{2m}}(a_{2m})A & \text{if } g\notin A, \hspace{0.15cm} k=2m, \\
		A & \text{ if } g=e_K.
	\end{cases} 
	\]
	This map induces the required homeomorphism $\phi$. As discussed in the proof of Theorem \ref{homeo free product}, for every element in $\delta_{Stab}(\Gamma_f)$, there exists a unique reduced representative $[g,\xi]$ where $g=a_1h_{i_1}a_2h_{i_2}\dots  a_{n}h_{i_k}$.
	
	Define the map $F : \partial_f (A*F)\to \partial_f(A*F)$ in the following manner:
	\begin{enumerate}
		\item \[
		F([g,\xi]):=\begin{cases} [\varphi_{i_1}(a_1)\varphi'_{i_2}(a_2)\dots \varphi_{i_{2m-1}}(a_{2m-1}) ,\xi] & \text{ if } g\notin A, \hspace{0.15cm} k=2m-1, \\
			[\varphi_{i_1}(a_1)\varphi'_{i_2}(a_2)\dots \varphi_{i_{2m-1}}(a_{2m-1})\varphi'_{i_{2m}}(a_{2m}),\xi] & \text{ if }  g\notin A, \hspace{0.15cm} k=2m, \\
			[e_K,\xi] & \text{ if } g=e_K.
		\end{cases}
		\]
		\item $F(\eta):= \phi(\eta) \text{ where } \eta \in \partial K.$
	\end{enumerate}
	From Remark \ref{rem1}(2) and Corollary \ref{main cor}, we can deduce that $F$ is continuous at every $\xi \in \delta_{Stab}(\Gamma_f)$. It is straightforward to check that $F$ is also continuous at every point in $\partial T_1$. Hence, we conclude that $F$ is continuous. Since $\overline{\Gamma}_f$ is compact and $F$ is a bijection, $F$ is also a homeomorphism. Hence, we conclude the proof.
\end{proof}
\section{Homeomorphism types of Floyd boundaries of amalgamated free product and HNN extension}\label{5}
So far, we have proved that homeomorphism types of Floyd boundaries of free products of infinite groups depends only on homeomorphism types of Floyd boundary of free factors. Now, we proceed to prove similar results for groups admitting a decomposition as amalgamted free products and HNN extension over finite groups. We fix a Floyd function $f$ for this section.
\subsection{Amalgam case}
The main result of this subsection is the following:
\begin{theorem}\label{amalgam}
	Suppose $A$ is an infinite group and $K =A*_{H}B$, $G=A*B$. Then, we have the following: 
	\begin{enumerate}
		\item $\partial_f(K) \simeq \partial_f(G), \text{ if } B \text{ is also infinite},$
		\item $\partial_f(K) \simeq \partial_f (A*A), \text{ if } B \text{ is finite and } B \neq H,$
		\item $\partial_f (B*_H C) $ is homeomorphic to the Cantor set if $B,C$ are non-trivial finite groups such that $B\ast_H C$ has infinitely many ends.
	\end{enumerate}
\end{theorem}
\begin{proof}
	Denote the scaled Cayley graph of $G$ and $K$ by $\Gamma_f$ and $\Gamma'_f$, and Bass--Serre trees by $T$ and $T'$, respectively.
	Let $\tau$ and $\tau'$ be the edges joining $A$ and $B$ in $T$ and $T'$, respectively. Let $A_1\subset A$, $B_1 \subset B$ be subsets that contain exactly one coset representative for each coset of $H$ in $A$ and $B$, respectively. We also assume that identity is contained in both of them. For each edge $g\tau'\in T'$, we have the unique reduced form of $g$ such that $g=a'_1b'_1a'_2b'_2\dots a'_n b'_n$, where $a'_j \in A_1-\{1\}$ and $b'_j \in B_1-\{1\}$, allowing $a'_1 =1$ and $b'_n=1$.

	{\bf Case 1.} If $B$ is infinite, then $\partial_f B$ is not empty.  It is easy to check that $A_1$ and $B_1$ are countable, dense subsets of $A \cup \partial_f A$ and $B \cup \partial_f B$, respectively. By applying Corollary \ref{topo1}, we get bijections $\varphi: A \rightarrow A_1$ and $\psi: B \rightarrow B_1$ with $\varphi(1)=1$ and $\psi(1)=1$, that extend to homeomorphisms $\varphi \cup \text{Id}: A \cup \partial_f A \rightarrow A_1 \cup \partial_f A$ and $\psi \cup \text{Id}: B\cup \partial_f B \rightarrow B_1 \cup \partial_f B$.
	
	Define $\iota: T \rightarrow T'$, such that \[\iota(a_1b_1a_2b_2\dots a_n b_n \tau)=\varphi(a_1)\psi(b_1)\varphi(a_2)\psi(_2)\dots \varphi(a_n)\psi(b_n)\tau',\] where $a_1b_1a_2b_2\dots a_nb_n$ is the reduced form of an element in $G$. The map $\iota$ is a graph isomorphism inducing homeomorphism $\bar{\iota}: \overline{T} \rightarrow \overline{T'}$. We will define a map $F: \delta(\Gamma_f) \rightarrow \delta(\Gamma'_f)$ as follows:
	\begin{enumerate}
		\item Following the discussion in the proof of Theorem \ref{homeo free product}, for every element in $\delta_{Stab}(\Gamma_f)$, we choose the reduced representative $[g, \xi]$ with $g$ having the reduced form $ a_1b_1\dots a_nb_n$. Define
		\[
		F([g,\xi]):= [\varphi(a_1)\psi(b_1)\varphi(a_2)\psi(b_2)\dots \varphi(a_n)\psi(b_n),\xi].
		\]
		\item For $\eta \in \partial T$, define $F(\eta):=\bar{\iota}(\eta).$
	\end{enumerate}
	Assume that $H=\{e, h_1,h_2, \dots,h_l\}$.
	For $1\leq k\leq l$, let $\varphi_k: A \rightarrow A_1h_k $ be the maps that take $a$ to $\varphi(a)h_k.$ Since $A= A_1 \sqcup A_1h_k$, by Remark \ref{rem1}(2), we can show that $F$ is continuous at every point of $\partial_f(G_v)$, for each vertex $v \in T$ that represents a coset of A. Applying a similar argument for $B$ gives the conclusion that $F$ is continuous at every point of $\delta_{Stab}(\Gamma_f)$. Since  $\overline{\iota}$ is a homeomorphism, it is easy to see that $F$ is also continuous at every point of $\partial T$. Since $F$ is a continuous bijection and $\overline{\Gamma}_f$ is compact, $F$  is a homeomorphism.
	
	{\bf Case 2.} If $B$ is a finite group then $\partial_f B$ is empty. Let $m$ be the index of $H$ in $B$. We show that $ \partial_f(A*\mathbb{Z}_m)\simeq \partial_f(K).$ We will use the map $\varphi$ defined in the last case. Let $\psi_1: \mathbb{Z}_m \rightarrow B_1$ be a bijection with $\psi_1(1)=1$. Let $\tau_1$ be the edge joining $A$ and $\mathbb{Z}_m$ in the Bass--Serre tree $T_1$ of $A*\mathbb{Z}_m$. Define $\iota_1: T_1 \rightarrow T'$, that takes the edge $g\tau_1$ to $\varphi(a_1)\psi_1(b_1)\varphi(a_2)\psi_1(b_2)\dots \varphi(a_n)\psi_1(b_n)\tau'$, where $a_1b_1a_2b_2\dots a_nb_n$ is the reduced form of $g \in A*\mathbb{Z}_m$. Note that $\iota_1$ is an isomorphism and hence it induces a homeomorphism $\overline{\iota}_1: \overline{T}_1\rightarrow \overline{T'}$. Now, $\bar{\iota}_1$ can be used to define a homeomorphism between $\partial_f(K)$ and $\partial_f(A*\mathbb{Z}_m)$ as we did in the last case. Finally, by Proposition \ref{finite amalgam}, we get that 
	\[
	\partial_f{K} \simeq \partial_f(A*\mathbb{Z}_m) \simeq \partial_f(A*A),
	\] 
	which completes the proof.
	
	{\bf Case 3.} Let $L:= B*_H C$. We denote the Bass--Serre tree and Cayley Graph of $L$ by $T''$ and $\Gamma_f''$, respectively. From Section \ref{3}, we know that $\partial_f(L) \simeq \delta_{Stab}\Gamma_f'' \cup \partial T''$. Since, $B$ and $C$ are finite groups, $\partial_f(B)$ and $\partial_f(C)$ are empty, which implies that $\delta_{Stab}(\Gamma_f'')$ is empty. Hence, $\partial_f(L) \simeq \partial T''$. Since the degree of the vertices of $T''$ are bounded above by a fixed natural number, $\partial T''$ is homeomorphic to the Cantor set.
\end{proof}
\begin{rem}
	If $B\ast_H C$, in Theorem \ref{amalgam}(3), is $2$-ended then $G$ is virtually cyclic and hence $|\partial_f(G)|=2$ (see \cite[Proposition 7]{karlsson}).
\end{rem}
We immediately have the following corollary whose proof is an easy application of induction:
\begin{cor}\label{homeo tree of groups}
	Let $Y$ be a tree with $k>1$ vertices. Let $(\GG,Y)$ be a graph of groups over $Y$ such that each edge group is finite. Suppose $G$ is the fundamental group of $(\GG,Y)$ and $G_1,G_2,\dots,G_k$ are the vertex groups. Then, we have the following:
	\begin{enumerate}
		\item If each $G_i$ is infinite then $\partial_f(G)\simeq\partial_f(G_1\ast\dots\ast G_k)$.
		\item If $G_j$ is infinite for some $j$ and, for each $i\neq j$, $G_i$ is finite. Then, $\partial_f(G)\simeq\partial_f(G_j\ast G_j)$.
		\item By reindexing (if required), suppose there exists $2\leq k'\leq k$ such that $G_i$ is infinite for $1\leq i\leq k'$ and $G_i$ is finite for $i>k'$. Then, $\partial_f(G)\simeq\partial_f(G_1\ast\dots\ast G_{k'})$.
		\item If $G_i$'s are finite and $G$ is infinite-ended then $\partial_f(G)$ is homeomorphic to the Canter set.
	\end{enumerate} 
\qed
\end{cor}
\subsection{HNN extension Case}
This subsection is devoted to prove the following theorem:
\begin{theorem}\label{hnn}
	Let $G=A*_{F\simeq F'}$, where $F$ and $F'$ are finite subgroups. Then, we have the following:
	\begin{enumerate}
		\item If $A$ is an infinite group then $\partial_f(G)\simeq \partial_f(A*A)$.
		\item If $A$ is finite and either $F$ or $F'$ is not equal to $A$ then $\partial_f(G)$ is homeomorphic to the Cantor set.
	\end{enumerate}
\end{theorem}
\begin{proof}
	(1) For the sake of distinction, let $K=A*A'$ where $A'$ is an isomorphic copy of $A$. Let $S_A$ be a generating set for $A$. We denote the Bass--Serre tree and Cayley graph with respect to generating set $S_A \cup S_A'$ of $K$ by $T'$ and $\Gamma'_f$, respectively. From the discussion in Subsection \ref{hnn construnction}, it is sufficient to produce a homeomorphism between $\delta(\Gamma_f)$ and $\delta(\Gamma'_f)$. Let $A_1$ and $A_2$ be the set of left coset representatives of $F$ and $F'$ in $A$, respectively, such that $1 \in A_1\cap A_2$. For every coset of $A$ in $G$, there exists a unique reduced representative $g$ of that coset of the form $g=a_0t^{\varepsilon_0}a_1t^{\varepsilon_1}\dots a_{n}t^{\varepsilon_{n}}$, $\varepsilon_i=\pm 1$, $a_i \in A_k$ when $\varepsilon_i= (-1)^{(k+1)}$, where $k\in\{1,2\}$, and for any $i\geq 1$, $\varepsilon_{i-1}= -\varepsilon_{i}$, $a_i=1$ is not possible. 
	
	For this proof, we will denote the identity map on $\partial_f(A)$ by Id.
	From Lemma $\ref{main topo lemma}$, we have disjoint subsets $B_1,B_2 \subset A$ such that $A=B_1\sqcup B_2$, with bijections $\psi_i: A \rightarrow B_i$ such that $\psi_i \cup \text{Id}$ are homeomorphisms. Since $A_i$'s are dense in $A_i \cup \partial_f(A)$, there exist bijections $\alpha_i: A_i \rightarrow A$ such that $\alpha_i \cup \text{Id}$ are homeomorphisms (see Corollary \ref{topo1}). Define a map $\varphi_i: A_i \rightarrow B_i$ as $\varphi_i:= \psi_i\circ\alpha_i$. Similarly, define the maps $\varphi'_i: A_i \to B'_i$, where $B_i'$'s are the isometric copy of $B_i$ in $A'$.  By construction, we conclude that $\varphi_i \cup \text{Id}$ and $\varphi_i'\cup\text{Id}$ are homeomorphisms.
	
	Firstly, we define a graph isomorphism $\iota: T \to T'$ as follows:
	\[
	\iota(gA):= \begin{cases}
		\varphi'_{\varepsilon_0}(a_0)\varphi_{\varepsilon_1}(a_1)\dots \varphi'_{\varepsilon_n}(a_n)A & \text{ if } n \text{ is even,}
		\\
		\varphi'_{\varepsilon_0}(a_0)\varphi_{\varepsilon_1}(a_1)\dots \varphi_{\varepsilon_n}(a_n)A' & \text{ if } n \text{ is odd,}
		\\ A  & \text{ if } g=e_G.
	\end{cases}
	\] which induces the homeomorphism $\bar{\iota}: \overline{T} \to \overline{T'}$. Define the map $F: \delta(\Gamma_f) \to \delta(\Gamma'_f)$ as follows:
	\begin{itemize}
		\item Let $\xi \in \partial_f(\Gamma^f_{gA})$, then define 
		\[
		F([g,\xi]):= \begin{cases}
			[\varphi'_{\varepsilon_0}(a_0)\varphi_{\varepsilon_1}(a_1)\dots \varphi'_{\varepsilon_n}(a_n),\xi] & \text{ if } n \text{ is even,}
			\\
			[\varphi'_{\varepsilon_0}(a_0)\varphi_{\varepsilon_1}(a_1)\dots \varphi_{\varepsilon_n}(a_n),\xi] & \text{ if } n \text{ is odd}
			\\
			[e_G, \xi] & \text{ if } g= e_G.
		\end{cases} 
		\]
		\item If $\eta \in \partial T$, then define 
		$F(\eta):= \bar{\iota}(\eta)$.
		
	\end{itemize}
	Using Remark \ref{rem1}(2), we can show that $F$ is continuous. Since $\delta(\Gamma_f)$ is compact and $F$ is a bijection, we conclude that $F$ is a homeomorphism.
	
	(2) Since $A$ is finite, $\delta_{Stab}(\Gamma_f)$ is empty. Hence, $\partial_f(G)\simeq \partial T$. However, as $T$ is $k$-regular where $k\geq 3$, $\partial T$ is homeomorphic to the Cantor set. 
\end{proof}
\begin{rem}
	In Theorem \ref{hnn}(2), if $A=F$ then $A=F'$ too as $A$ is finite. In this case, $G$ is virtually cyclic and hence $|\partial_f(G)|=2$.
\end{rem}
\begin{cor}\label{homeo free product cyclic}
	For $n\geq 2$, let $G=A_1\ast\dots\ast A_n$ be a free product of elementary groups such that at least one free factor is infinite. Then, the Floyd boundary of $G$ is homeomorphic to the Cantor set.
\end{cor}
\begin{proof}
	There are two cases to be consider:
	
	{\bf Case 1.} By reindexing, if required, assume that $A_1$ is infinite and $A_i$ is finite for $2\leq i\leq n$. Then, by Corollary \ref{homeo tree of groups}(2), $\partial_f(G)\simeq \partial_f(A_1\ast A_1)$. Since $A_1$ is virtually cyclic, $\partial_f(A_1)\simeq \partial_f(\mathbb Z)$. Thus, by Theorem \ref{homeo free product}, $\partial_f(A\ast A)\simeq \partial_f(\mathbb Z\ast\mathbb Z)$. Let $B$ be a non-trivial finite group. Then, by Proposition \ref{finite amalgam}, $\partial_f(\mathbb Z\ast\mathbb Z)\simeq \partial_f(\mathbb Z\ast B)$. Since $B_{\ast_{\{1\}}}$ is isomorphic to $B\ast\mathbb Z$, their Floyd boundaries are homeomorphic. Now, by Theorem \ref{hnn}(2), $\partial_f(B\ast\mathbb Z)$ is homeomorphic to the Cantor set.
	
	{\bf Case 2.} By reindexing, if necessary, suppose there exists $2\leq k'\leq n$ such that $A_i$ is infinite for $1\leq i\leq k'$ and $A_j$ is finite for $j>k'$. Then, by Corollary \ref{homeo tree of groups}(3), $\partial_f(G)\simeq\partial_f(A_1\ast\dots\ast A_{k'})$. As $A_i$'s are infinite virtually cyclic for $1\leq i\leq k'$, by induction, it is easy to prove that $\partial_f(G)$ is homeomorphic to the Cantor set. 
\end{proof}
We end this subsection by proving the following which remove the dependence of Floyd boundary of free products on virtually cyclic groups.

\begin{cor}\label{remove virtual cyclic}
	For $k\geq 2$, let $G=B_1\ast\dots\ast B_k$ be a free product of infinite groups such that, for $2\leq i\leq k$, $B_i$ is virtually cyclic. Then, $\partial_f(G)$ is homeomorphic to $\partial_f(B_1\ast B_1)$. 
\end{cor}
\begin{proof}
	We prove it by induction on $k$. Let $k=2$. Then, $\partial_f(G)\simeq\partial_f(B_1\ast\mathbb Z)$. Since $B_1\ast \mathbb Z$ is isomorphic to $B_{\ast_{\{1\}}}$, $\partial_f(G)\simeq\partial_f(B_{\ast_{\{1\}}})$ which is homeomorphic to $\partial_f(B_1\ast B_1)$ by Theorem \ref{hnn}(1). Suppose $\partial_f(B_1\ast\dots\ast B_{k-1})\simeq\partial_f(B_1\ast B_1)$. Now, by Theorem \ref{homeo free product}, $\partial_f(G)\simeq\partial_f(B_1\ast B_1\ast B_k)$ which is homeomorphic to $\partial_f(B_1\ast B_k)$ by Theorem \ref{homeo three product}. Now the corollary follows from the case when $k=2$.
\end{proof}
By combining Corollary \ref{homeo tree of groups}(3) and Corollary \ref{remove virtual cyclic}, we have the following:
\begin{cor}\label{main cor3}
 For $k,l\in\mathbb N$, let $G=A_1\ast\dots\ast A_k\ast B_1\ast\dots\ast B_l$ where $A_i$'s are non-elementary and $B_j$'s are elementary groups. Then, we have the following:
 
 (1) If $k\geq 2,l\geq 1$ then the Floyd boundary of $G$ is homeomorphic to $\partial_f(A_1\ast\dots\ast A_k)$.
 
 (2) If $k=1$ and $l\geq 1$ then $\partial_f(G)\simeq\partial_f(A_1\ast A_1)$. \qed
\end{cor}
\section{Main Theorem}\label{6}
Let $f$ and $f'$ be two Floyd functions which is fixed for this and the next section. Here, the main goal is to prove the following:
\begin{theorem}\label{general case}
	Suppose $G$ and $H$ are two finitely generated infinite groups. Suppose $\mathcal G$ and $\mathcal H$ are graphs of groups decomposition of $G$ and $H$, respectively such that the edge groups of $\mathcal G$ and $\mathcal H$ are finite. Then, we have the following:
	\begin{enumerate}
		\item If each vertex group of $\mathcal G$ is elementary and $G$ has infinitely many ends then $\partial_f(G)$ is homeomorphic to the Cantor set.
		
		\item Suppose at least one vertex group of either $\mathcal G$ or $\mathcal H$ is non-elementary. Let $h(\mathcal G)$ and $h(\mathcal H)$ denote the set of homeomorphism types of Floyd boundaries of non-elementary vertex groups of $\mathcal G$ and $\mathcal H$ with respect to $f$ and $f'$, respectively. If $h(\mathcal G)=h(\mathcal H)$ then $\partial_f(G)$ is homeomorphic to $\partial_{f'}(H)$.
	\end{enumerate}
	
\end{theorem}
To prove the above theorem, we use the following result:
\begin{theorem}\label{main theorem}
	For $n,m\geq 2$, let $G_1=A_1\ast\dots \ast A_n$ and $G_2=B_1\ast\dots \ast B_m$ be two free products where at least one free factor of either $G_1$ or $G_2$ is non-elementary. Let $h(G_1)$ and $h(G_2)$ denote the set of homeomorphism types of Floyd boundaries of non-elementary free factors of $G_1$ and $G_2$ with respect to $f$ and $f'$, respectively. If $h(G_1)=h(G_2)$ then the Floyd boundary of $G_1$ is homeomorphic to the Floyd boundary of $G_2$.
\end{theorem}

Before getting into the proof of Theorem \ref{main theorem}, we prove the following result which turns out to be useful in proving the above theorem:
\begin{theorem}\label{homeo three product}
	Suppose that $G=A\ast B\ast C$ is a free product of infinite groups. If $\partial_f(A)$ is homeomorphic to $\partial_f(C)$ then $\partial_f(G)$ is homeomorphic to $\partial_f(A\ast B)$.
\end{theorem}

We immediately have the following corollaries:
\begin{cor}\label{homeo four}
	Suppose that $G=A\ast B\ast C\ast D$ is a free product of infinite groups. If $\partial_f(A)\simeq\partial_f(C)$ and $\partial_f(B)\simeq\partial_f(D)$ then $\partial_f(G)\simeq\partial_f(A\ast B)$.
\end{cor}
\begin{proof}
	By applying Theorem \ref{homeo three product}, we see that $\partial_f(G)\simeq\partial_f(A\ast B\ast D)$ which is again homeomorphic to $\partial_f(A\ast B)$ by Theorem \ref{homeo three product}.
\end{proof}

\begin{cor}\label{same group homeo}
	For $n\geq 2$, suppose $A=A_1\ast\dots \ast A_n$ is a free product of infinite groups. Suppose $2\leq i\leq n$, $\partial_f(A_i)$ is homeomorphic to $\partial A_1$. Then, $\partial_f(A)$ is homeomorphic to $\partial_f(A_1\ast A_1)$.
\end{cor}
\begin{proof}
	We prove it by induction on $n$. From Theorem \ref{homeo three product}, it follows that $\partial_f(A_1\ast A_2\ast A_3)\simeq\partial_f(A_1\ast A_1).$ Suppose $\partial_f(A_1\ast\dots \ast A_{n-1})\simeq\partial_f(A_1\ast A_1)$. Then, by Theorem \ref{homeo free product}, $\partial_f(A)\simeq \partial_f(A_1\ast A_1\ast A_1)$ which is homeomorphic to $\partial_f(A_1\ast A_1)$.
\end{proof}
Using Theorem \ref{homeo three product} and Proposition \ref{finite amalgam}, we immediately have the following corollary:
\begin{cor}\label{cr1}
	Let $G=A_1*B_1$ and $H=A_2*B_2$, where $A_1$ and $A_2$ are infinite groups while $B_1$ and $B_2$ can be finite such that $\partial_f(A_1) \simeq \partial_f (A_2)$ and $\partial_f(B_1)\simeq \partial_f(B_2)$. Then $\partial_f(G) \simeq \partial_f (H).$  
	\qed
\end{cor}

\begin{cor}\label{cr2}
	Let $G=A\ast F_1\dots \ast F_n$ be a free product of groups such that $A$ is an infinite group and $F_i$'s are finite groups. Then $\partial_f(G) \simeq \partial_f(A*A)$.
\end{cor}
\begin{proof}
	We prove it by induction on $n$. For $n=1$, it is Proposition \ref{finite amalgam}. Let $H=A\ast F_1\ast\dots*F_{n-1}$. Assume that $\partial_f(H)\simeq\partial_f(A\ast A)$.  Now, by Proposition \ref{finite amalgam}, we see that $\partial_f(G)\simeq \partial_f(H\ast H)$. By inductive hypothesis, we have that $\partial_f(H)\simeq \partial_f(A\ast A)$. Hence, by Theorem \ref{homeo free product}, $\partial_f(G)\simeq \partial_f(A\ast A\ast A\ast A)$. Finally, by Corollary \ref{homeo four}, we have that $\partial_f(G)\simeq \partial_f(A\ast A)$. 
\end{proof}

\subsection{Proof of Theorem \ref{homeo three product}} Let $G$ be as in Theorem \ref{homeo three product} and let $T$ be its Bass--Serre tree. Suppose $G'=A\ast B$ and let $T'$ be its Bass--Serre tree. In this subsection, first of all, we prove that there is an isomorphism between $T$ and $T'$. Using this isomorphism, we define a map from $\partial_f(G)$ to $\partial_f(G')$ and prove that this is in fact a homeomorphism. 
\vspace{.2cm}

{\bf Isomorphism between $T$ and $T'$:} Let $\Gamma_B$ be the Cayley graph of $B$ with respect to a finite generating set. Then, $\overline{\Gamma}_B=\Gamma_B\cup\partial_f(B)$ is a compact metrizable space. With restricted metric, $B\cup\partial_f(B)$ is compact, $B$ is dense in $\overline{\Gamma}_B$ and $B$ contains all isolated points of $\overline{\Gamma}_B.$ Then, by Lemma \ref{main topo lemma}, there exist $B_1,B_2$ such that $B\setminus\{e_B\}=B_1\sqcup B_2$ and there exist bijections $\beta_1:B\to B_1,\beta_2:B\to B_2$ that induce homeomorphisms $\bar{\beta_1}:B\cup\partial_f(B)\to B_1\cup\partial_f(B)$ and $\bar{\beta_2}:B\cup\partial_f(B)\to B_2\cup\partial_f(B)$, respectively. Since $\partial_f(A)$ is homeomorphic to $\partial_f(C)$, by Lemma \ref{main lemma}, there exists a bijection $\gamma:C\to A$ such that $\gamma(1)=1$ and it induces a homeomorphism from $A\cup\partial_f(A)$ to $C\cup\partial_f(C)$. Let $\alpha$ be the identity map from $A$ to $A$. Any element
$g\in A*B*C$ can be written as
$g= a_1 b_1d_1b_2d_2\dots b_n d_n$, where $a_1 \in A$, $b_i \in B$, and $d_i \in A\setminus\{e_A\} \cup C\setminus\{e_C\} $, and $a_1,b_i,d_n$ can be the identity in their respective groups. We shall call this the {\em standard form} of $g$.

Let $\tau$ be the edge between $A$ and $B$ in $T'$, and $\tau_1$, $\tau_2$ be the edges between $A$ and $B$, $B$ and $C$, respectively in $T$. Note that any edge in $T'$ is a $G'$-translate of $\tau$. Similarly, any edge of $T$ is either a $G$-translate of $\tau_1$ or a $G$-translate of $\tau_2$.  Define 
a map $\phi:T\to T'$ in the following manner:

{\bf Case 1.} Let $g\tau_1$ be an edge in $T$ for some $g\in G$ and $a_1b_1d_1,\dots ,b_nd_n$ be the standard form for $g$. There are two subcases:

{\bf Subcase 1(a).} Suppose $d_n= 1$. Then, we define
$$h:= a_1 \beta_{i_1}(b_1) \delta_1(d_1),\dots ,\beta_{i_n}(b_n)$$ where for all $1\leq k<n$, $i_k=1$ if $d_k \in A\setminus\{e_A\}$ otherwise $i_k=2$, and $\delta_{k}= \alpha$ if $d_k \in A\setminus\{e_A\}$ otherwise  $\delta_k= \gamma$, and $i_n =1$. In this case, we define $\phi(g\tau_1)=h\tau$. 

{\bf Subcase 1(b).} Suppose $d_n\neq 1$. Then, $d_n\in A$ and we define  
$$h:= a_1 \beta_{i_1}(b_1) \delta_1(d_1),\dots ,\beta_{i_n}(b_n) \delta_n(d_n)$$ where, for $1\leq k\leq n$, $i_k=1$ if $d_k \in A\setminus\{e_A\}$ otherwise $i_k=2$, and $\delta_{k}= \alpha$ if $d_k \in A\setminus\{e_A\}$ otherwise  $\delta_k= \gamma$. In this case, we define $\phi(g \tau_1)=h\tau$.

{\bf Case 2.} Let $g\tau_2$ be an edge in $T$ for some $g\in G$ and $a_1b_1d_1,\dots ,b_nd_n$ be the standard form for $g$. Again there are two subcases:

{\bf Subcase 2(a).} Suppose $d_n\neq 1$. Then, we define
$$h:= a_1 \beta_{i_1}(b_1) \delta_1(d_1),\dots ,\beta_{i_n}(b_n) \delta_n(d_n)$$
where, for $1\leq k\leq n$, $i_k=1$ if $d_k \in A\setminus\{e_A\}$ otherwise $i_k=2$, and $\delta_{k}= \alpha$ if $d_k \in A\setminus\{e_A\}$ otherwise  $\delta_k= \gamma$. In this case, we define $\phi(g\tau_2)=h\tau$. 

{\bf Subcase 2(b).} Suppose $d_n=1$. Then, we define
$$h:= a_1 \beta_{i_1}(b_1) \delta_1(d_1),\dots ,\beta_{i_n}(b_n)$$
where for all $1\leq k<n$, $i_k=1$ if $d_k \in A\setminus\{e_A\}$ otherwise $e_k=2$, and $\delta_{k}= \alpha$ if $d_k \in A\setminus\{e_A\}$ otherwise $\delta_k= \gamma$, and $i_n=2$. In this case, we define $\phi(g\tau_2)=h\tau$.

From the construction of $\phi$, it follows that $\phi$ is an isomorphism.
\vspace{.2cm}

{\bf Homeomorphism from $\partial_f(G)$ to $\partial_f(G')$:} Let $\Gamma_f$ and $\Gamma_f'$ be the scaled Cayley graphs of $G$ and $G'$, respectively and let $\Gamma_f$ and $\Gamma'_f$
be compact metrizable spaces as constructed in Section \ref{3}. From Proposition \ref{p2}, it follows that $\partial_f(G)$ is homeomorphic to $\delta(\Gamma_f)$ and $\partial_f(G')$ is homeomorphic to $\delta(\Gamma'_f)$. Let $h_A:\partial_f(A)\rightarrow \partial_f(A),$ $h_B: \partial_f(B) \rightarrow \partial_f(B)$ be the identity
homeomorphisms, and let $h_C : \partial_f(C)\rightarrow \partial_f(A)$ be a homeomorphism such that $\gamma\cup h_C: C \cup \partial_f(C)\rightarrow A \cup \partial_f(A)$ is a homeomorphism. Since $\phi: T \rightarrow T'$ is an isomorphism, it induces a homeomorphism $\partial \phi: \partial T \rightarrow \partial T'$.

We define a map $F: \delta(\Gamma_f) \rightarrow \delta(\Gamma'_f)$ in the following manner:
\begin{enumerate}
	\item  Suppose $\xi \in \delta_{Stab}(\Gamma_f)$. Let $v$ be a vertex in $T$ such that $\xi \in \partial_f(G_v)$. As discussed in the proof of Theorem \ref{homeo free product}, for $\xi$, there exists a unique \textit{reduced representative} $[g,\xi]$. Define
	\[ F([g,\xi]):=\begin{cases} 
		[\phi(g),\xi] & \text{ if } v=gA,  \\
		[\phi(g), \xi] & \text{ if } v=gB, \\
		[\phi(g),h_C(\xi)] & \text{ if } v=gC,
	\end{cases}
	\] 
	\item Suppose $\eta \in \partial T$. Then, define $F(\eta)= \partial \phi(\eta)$.
\end{enumerate}
From Remark \ref{rem1}(2) and Corollary \ref{main cor}, the continuity of $F$ on every point of $\delta_{Stab}(\Gamma_f)$ follows. From the definition of neighborhoods, it is straightforward to check the continuity of $F$ at every point of $\partial T$. Since 
$\delta(\Gamma_f)$ is compact and  $F$ is a bijection, we conclude that $F$ is a homeomorphism. Hence, we complete the proof of Theorem \ref{homeo three product}.
\qed

\subsection{Proof of Theorem \ref{main theorem}}

Let $n_1<n_2<\dots <n_k<n_{k+1}$ be the natural numbers such that $n=n_1+\dots +n_k+n_{k+1}$. By reindexing the groups, if required, we have that $G_1=(A_1\ast\dots \ast A_{n_1})\ast\dots \ast (A_{n_{k-1}+1}\ast\dots \ast A_{n_k})\ast\dots \ast (A_{n_{k}+1}\ast\dots \ast A_{n_{k+1}})$ such that $\partial_f(A_{s_j})\simeq\partial_f(A_{n_j})$, where $1\leq j\leq k$ and $1+n_{j-1}\leq s_j\leq n_j$ with $n_0=0$, and $A_i$'s are elementary groups for $n_{k+1}+1 \leq i \leq n_{k+2}$. Now, the following two cases can occur:

{\bf Case 1.} Suppose $k=1$. If $n_1=1$ too then, by Corollary \ref{main cor3}(2), $\partial_f(G_1)\simeq\partial_f(A_1\ast A_1)$. Since $h(G_1)=h(G_2)$, by reindexing if necessary, $\partial_f(G_2)\simeq\partial_{f'}(B_1\ast B_1)$. Then, by Theorem \ref{homeo free product}, we are done. If $n_1\geq 2$ then, by Corollary \ref{main cor3}(1), $\partial_f(G_1)\simeq\partial_f(A_1\ast\dots\ast A_{n_1})$. By Corollary \ref{same group homeo}, $\partial_f(G_1)\simeq\partial_f(A_{n_1}\ast A_{n_1})$. As $h(G_1)=h(G_2)$, there exists $m_1$ such that $\partial_{f'}(G_2)\simeq\partial_{f'}(B_{m_1}\ast\dots B_{m_1})$. Again, we are done by Theorem \ref{homeo free product}.

 {\bf Case 2.} Suppose $k\geq 2$. By Corollary \ref{main cor3}(1), we get $\partial_f (G_1) \simeq \partial_f(A_1*A_2*\dots A_{n_k})$. Using Corollary \ref{same group homeo} and Corollary \ref{gen free product}, we see that $\partial_f(G_1)\simeq \partial_f(A_{n_1}\ast A_{n_1}\ast A_{n_2}\ast A_{n_2}\ast\dots \ast A_{n_k}\ast A_{n_k})$. An easy application Corollary \ref{homeo four} and induction on $k$ gives that $\partial_f(G_1)\simeq \partial_f(A_{n_1}\ast\dots \ast A_{n_k})$. Since $h(G_1)=h(G_2)$, by following the same process as above, there exist $m_1,m_2,\dots ,m_k$ such that $\partial_{f'}(G_2)\simeq\partial_{f'}(B_{m_1}\ast \dots \ast B_{m_k})$. Finally, by Corollary \ref{gen free product}, we have that $\partial_f(G_1)\simeq\partial_{f'}(G_2)$. \qed

\subsection{Proof of Theorem \ref{general case}}
	Let $Y_G$ and $Y_H$ be the undeline graphs of $\mathcal G$ and $\mathcal H$. Let $T_G$ and $T_H$ be the maximal trees in $Y_G$ and $Y_H$, respectively. Then, the restriction of $\mathcal G$ and $\mathcal H$ to $T_G$ and $T_H$ are trees of groups with finite edge groups whose fundamental groups are, say, $G'$ and $H'$.
	
	(1) 
	If $G'$ is infinite then, by Theorem \ref{hnn}, $\partial_f(G)\simeq \partial_f(G'\ast G')$ which is homeomorphic to the Cantor set by Corollary \ref{homeo tree of groups}(3,4) and Corollary \ref{homeo free product cyclic}. Suppose $G'$ is finite. Since $G$ is repeated HNN extension of $G'$ and $G$ has infinitely many ends, at least one edge outside $T_G$ gives an HNN extension, say, $G''$ such that it is an infinite group. Then, by repeated application of Theorem \ref{hnn}(1), we see that $\partial_f(G)\simeq\partial_f(G''\ast G'')$. Again, by Theorem \ref{hnn}(2), $\partial_f(G'')$ is homeomorphic to the Cantor set. Now, Theorem \ref{homeo free product} and Corollary \ref{homeo free product cyclic} give that $\partial_f(G)\simeq\partial_f(\mathbb Z\ast\mathbb Z\ast\mathbb Z\ast\mathbb Z)$ which is again homeomorphic to $\partial_f(\mathbb Z\ast\mathbb Z)$ by Corollary \ref{homeo four}. Hence, the Floyd boundary of $G$ is homeomorphic to the Cantor set. Similarly, by doing the same analysis, one can show that $\partial_{f'}(H)$ is homeomorphic to the Cantor set.
	
	(2) Let $G_1,\dots,G_k$ and $H_1,\dots,H_l$ be the vertex groups in $\mathcal G$ and $\mathcal H$, respectively. There are two cases to be consider:
	
	{\bf Case 1.} By reindexing (if required), we assume that $G_1$ is non-elementary and, for $i\neq 1$, $G_i$ is elementary. Then, $\partial_f(G')\simeq \partial_f(G_1\ast G_1)$ by Corollary \ref{main cor3}. Since $h(\mathcal G)=h(\mathcal H)$, by reindexing (if required), we assume that $H_1$ is non-elementary and, for $j\neq 1$, $H_j$ is elementary. Thus, $\partial_{f'}(H')\simeq\partial_{f'}(H_1\ast H_1)$. Now, Theorem \ref{hnn}(1) and Corollary \ref{homeo four} give us that $\partial_f(G)\simeq\partial_f(G_1\ast G_1)$. Similarly, $\partial_{f'}(H)\simeq\partial_{f'}(H_1\ast H_1)$. Hence, $\partial_f(G)\simeq\partial_f(H)$ by Theorem \ref{homeo free product}. 
	
	{\bf Case 2.} By reindexing (if required), we assume that there exists $2\leq k'\leq k$ such that $G_i$ is non-elementary for $1\leq i\leq k'$ and $G_i$ is elementary $k'+1\leq i\leq k$. Similarly, we assume that there exists $2\leq l'\leq l$ such that $H_j$ is non-elementary for $1\leq j\leq l'$ and $H_j$ is elementary for $l'+1\leq j\leq l$. Thus, by combining Corollary \ref{homeo tree of groups}(3) and Corollary \ref{remove virtual cyclic}, we have $\partial_f(G')\simeq\partial_f(G_1\ast\dots\ast G_{k'})$ and $\partial_{f'}(H)\simeq\partial_{f'}(H_1\ast\dots\ast H_{l'})$. Now, Theorem \ref{hnn}(1) gives that $\partial_f(G)\simeq\partial_f((G_1\ast\dots\ast G_{k'})\ast(G_1\ast\dots\ast G_{k'}))$ and $\partial_{f'}(H)\simeq\partial_{f'}((H_1\ast\dots\ast H_{l'})\ast(H_1\ast\dots\ast H_{l'}))$. Since $h(\mathcal G)=h(\mathcal H)$, Theorem \ref{main theorem} gives us that $\partial_f(G)\simeq \partial_{f'}(H)$. This complete the proof of the theorem.
\qed

	\begin{comment}

	Thus, by Theorem \ref{amalgam}, $\partial_f(G')\simeq\partial_f(G_1\ast\dots\ast G_k)$ and $\partial_f(H')\simeq\partial_f(H_1\ast\dots\ast H_l)$. Each edge in the complement of $T_G$ gives an HNN extension of $G$ over finite group. Hence, by Theorem  \ref{hnn}, $\partial_f(G)$ is homeomorphic to $\partial_f((G_1\ast\dots\ast G_k)\ast (G_1\ast\dots\ast G_k))$. Similarly, $\partial_f(H)$ is homeomorphic to $\partial_f((H_1\ast\dots\ast H_l)\ast (H_1\ast\dots\ast H_l))$. If $k=1$ then $\partial_f(G)\simeq \partial_f(G_1\ast G_1)$. There are two cases:
	\begin{itemize}
		\item Suppose $l=1$. Then, $\partial_f(H)\simeq\partial_f(H_1\ast H_1)$ and therefore, by Theorem \ref{homeo free product}, $\partial_f(G)\simeq\partial_f(H)$.
		\item Suppose $l\geq 2$. In this case, by Theorem \ref{main theorem}, $\partial_f(H)\simeq\partial_f(H_1\ast\dots\ast H_l)$. Since $h(\mathcal G)=\mathcal (H)$, by appying Theorem \ref{main theorem}, we see that $\partial_f(G)\simeq\partial_f(H)$.
	\end{itemize}
A similar analysis can be done when $l=1$ and $k\in\mathbb N$. Hence the theorem.

\end{comment}

\section{Components of Floyd boundary of infinite-ended groups}\label{7}
It is well known that Gromov boundary of a hyperbolic group is connected if and only if it is $1$-ended. In this section, first of all, we prove an analog of this fact in the context of Floyd boundary. Let $G$ be a $1$-ended group and $f$ be a Floyd function. We denote the scaled Cayley graph of $G$ with respect to a finite generating set by $\Gamma_f$ with the Floyd metric $d_f$, and $d$ denotes the metric on the Cayley graph. Assume that $\overline{\Gamma}_f$ is the metric completion with the metric $\bar{d}_f$. 

\begin{prop}
	$\partial_f (G)$ is connected.
\end{prop}
\begin{proof}
	Define a map $\pi: G \to \partial_f (G)$ such that  $\bar{d}_f(g, \pi(g))=\bar{d}_f(g, \partial_f(G))$. For the sake of contradiction, assume that $\partial_f(G)= V_1 \sqcup V_2$, where $V_i$'s are non-empty, disjoint open subsets of $\partial_f(G)$. Let $B_i:= \pi^{-1}(V_i)$ for $i=1,2$. Denote the closure of $B_i$ in $\overline{\Gamma}_f$ by $\overline{B}_1$. From definition of $\pi$, it follows that if $\{g_i\}$ is a sequence converging to a point $\xi\in\partial_f(G)$ then $\pi(g_i)$ also converges to $\xi$. Thus, it is see to check that $\overline{B}_i\subset B_i \cup V_i$.

	Let $\xi \in V_1$ and let $\{g_n\}$ be a sequence in $\Gamma_f$ such that $g_n \to \xi$. Suppose there exists an infinite subsequence of $\{g_{n}\}$ which is contained in $B_2$. Then $\xi \in \overline{B}_2$ which in turn implies that $\xi \in V_2$. This gives a contradiction as $V_1$ and $V_2$ are disjoint. Hence, $\{g_n\}$ is eventually contained in $B_1$. Applying a similar argument to $\xi' \in V_2$ leads us to the following conclusion:
	\[
	\overline{B}_i: = B_i \cup V_i.
	\]

	Since $V_1$ and $V_2$ are non-empty, there exists $\xi_1 \in V_1,$ $\xi_2 \in V_2$. Assume that $\{x_n\}$ and $\{y_n\}$ are sequences in $B_1$ and $B_2$, respectively, such that $x_n \to \xi_1$ and $y_n \to \xi_2$. For every $m\in \mathbb{N}$, denote $K_m$ the closed $d$-metric ball of radius $m$ about $e_G$. Since $G$ is $1$-ended, there exist subsequences $\{x_{n_m}\}$ and $\{y_{n_m}\}$ and paths $\gamma_{m}$ joining $x_{n_m}$ and $y_{n_m}$ such that $\gamma_m$ is contained in $\Gamma_f\setminus K_m$. Let $e_m=[a_m,b_m]$ be an edge such that $a_m\in B_1$ and $b_m\in B_2$. Since $\Gamma_f$ is compact, there exists subsequence $\{a_{m_k}\}$ that converges to $\eta \in V_1$. Since $d(a_{m_k},b_{m_k})=1$, it is easy to see that $b_{m_k} \to \eta$ in $\overline{\Gamma}_f$ too. This implies that $\eta \in V_2$, which gives us a contradiction. Hence, we have the desired result.    
\end{proof}
Now, we record the following fact whose proof is same as the proof of \cite[Proposition 6.3,6.4]{martin-swiat}.

\begin{prop}\label{connected components}
	Let $G$ be a group which splits as a graph of groups $\mathcal G$ such that the edge groups are finite and $G$ is infinite-ended. Suppose each non-elementary vertex group is $1$-ended. Then, each connected component is either homeomorphic to the Floyd boundary of a non-elementary vertex group or it is a singleton set. \qed
\end{prop}

Now, we are ready to prove the following which gives a partial converse of Theorem \ref{general case}.
\begin{theorem}\label{partial converse}
		Suppose $G$ and $H$ are two finitely generated infinite-ended groups. Suppose $\mathcal G$ and $\mathcal H$ are graphs of groups decomposition of $G$ and $H$, respectively such that the edge groups of $\mathcal G$ and $\mathcal H$ are finite, and each non-elementary vertex group of $\mathcal G$ and $\mathcal H$ is $1$-ended. Also, assume that at least one vertex group of $\mathcal G$ and $\mathcal H$ is non-elementary. Then, we have the following:
		
		Let $h(\mathcal G)$ and $h(\mathcal H)$ denote the set of homeomorphism types of Floyd boundaries of non-elementary vertex groups of $\mathcal G$ and $\mathcal H$, respectively. If $\partial_f(G)$ is homeomorphic to $\partial_{f'}(H)$ then $h(\mathcal G)=h(\mathcal H)$.
\end{theorem}
\begin{proof}
	Let $G_1,\dots,G_k$ and $H_1,\dots,H_l$ be the non-elementary vertex groups of $\mathcal G$ and $\mathcal H$, respectively. There are two cases to be consider:
	
	{\bf Case 1.} Suppose $k=1$ and $l=1$. Then, $\partial_f(G)\simeq \partial_f(G_1\ast G_1)$ and $\partial_{f'}(H)\simeq\partial_{f'}(H_1\ast H_1)$. Thus, $\partial_f(G_1\ast G_1)\simeq\partial_{f'}(H_1\ast H_1)$. Using Proposition \ref{connected components} and the fact that a homeomorphism takes connnected components to connected components, we see that $\partial_f(G_1)\sim\partial_{f'}(H_1)$. Suppose $k=1$ and $l\geq 2$. Then, $\partial_f(G)\simeq\partial_f(G_1\ast G_1)$ and $\partial_{f'}(H)\simeq\partial_{f'}(H_1\ast\dots\ast H_l)$. Again, by the same logic, we see that $\partial_f(G_1)\simeq\partial_{f'}(H_i)$ for $1\leq i\leq l$. A similar analysis can be done when $l=1$ and $k$ is any natural number.
	
	{\bf Case 2.} Suppose $k,l\geq 2$. In this case, we have that $\partial_f(G)\simeq\partial_f(G_1\ast\dots\ast G_k)$ and $\partial_{f'}(H)=\partial_{f'}(H_1\ast\dots H_l)$. Thus, $\partial_f(G_1\ast\dots\ast G_k)\simeq\partial_{f'}(H_1\ast\dots\ast H_l)$. Now, using Proposition \ref{connected components} and the fact that homeomorphism takes connected components to connected components, we see that $h(\mathcal G)=h(\mathcal H)$.
\end{proof}

In \cite{dunwoody}. M. Dunwoody proved that each finitely presented group $G$ has
a terminal splitting over finite subgroups, i.e. it is isomorphic to the fundamental
group of a graph of groups whose vertex groups are $1$-ended or finite
and whose edge groups are finite. Thus, by combining Theorem \ref{general case} and Theorem \ref{partial converse}, we have the following:
\begin{cor}
	Suppose $G$ and $H$ are two finitely presented infinite-ended groups. Suppose $\mathcal G$ and $\mathcal H$ are terminal splittings of $G$ and $H$, respectively such that at least one vertex of $\mathcal G$ and $\mathcal H$ is non-elementary. Then, we have the following:
	
	Let $h(\mathcal G)$ and $h(\mathcal H)$ denote the set of homeomorphism types of Floyd boundaries of non-elementary vertex groups of $\mathcal G$ and $\mathcal H$, respectively. Then, $\partial_f(G)$ is homeomorphic to $\partial_{f'}(H)$ if and only if $h(\mathcal G)=h(\mathcal H)$. \qed
\end{cor}

	We end the paper by noting the following:
\begin{cor}\label{virtually free cor}
	Let $G$ be a virtually free group. Then, for any Floyd function $f$, $\partial_f(G)$ is homeomorphic to the Cantor set. Conversely, if, for some Floyd function $f$, $\partial_f(G)$ is homeomorphic to the Cantor set then $G$ is virtually free.
\end{cor}
\begin{proof}
	Since virtually free groups admit a decomposition as graphs of finite groups, by Theorem \ref{general case}(1), $\partial_f(G)$ is homeomorphic to the Cantor set. Conversely, suppose that $\partial_f(G)$ is the Cantor set. This implies that $G$ is infinite-ended. Since virtually free groups are finitely presented, $G$ admits a terminal splitting where each vertex group is finite. This implies that $G$ is virtually free (see \cite{serre-trees}).
\end{proof}
\bibliography{homeo}
\bibliographystyle{amsalpha}
	
\end{document}